\documentclass[smallcondensed]{svjour3}       % onecolumn (second format)
\smartqed
\usepackage[dvips]{graphicx}
\usepackage[numbers]{natbib}
\usepackage{amsmath}
\usepackage{amssymb}
\usepackage{algorithm}
\usepackage{algorithmic}
\usepackage{bm}
\usepackage{color}
\usepackage{url}
\usepackage[pdftex,pdfstartview=FitH]{hyperref}
\usepackage{multirow}
\usepackage{enumerate}
\usepackage{pgf}
\usepackage{subcaption}

\newcommand{\qedwhite}{\hfill \ensuremath{\Box}}

\newcommand{\vect}{{\mbox{$\rm vec$}}}
\newcommand{\rank}{{\mbox{$\rm rank$}}}
\newcommand{\prox}{{\mbox{\boldmath $\rm prox$}}}
\newcommand{\proj}{{\mbox{\boldmath $\rm proj$}}}

\newcommand{\bvec}[1]{\mbox{\boldmath $#1$}}

\newcommand{\argmin}{\operatornamewithlimits{argmin}}
\newcommand{\Argmin}{\operatornamewithlimits{Argmin}}

\newcommand{\diag}{\mathrm{diag}}
\newcommand{\Diag}{\mathrm{Diag}}

\newtheorem{THEO}{Theorem}

\newtheorem{ALGo}{Algorithm}
\newtheorem{CONJ}{Conjecture}
\newtheorem{COND}{Condition}
\newtheorem{CORO}{Corollary}
\newtheorem{DEFI}{Definition}
\newtheorem{FACT}[THEO]{Fact}
\newtheorem{HYPO}[THEO]{Hypothesis}
\newtheorem{LEMM}{Lemma}
\newtheorem{PROB}{Problem}
\newtheorem{PROP}{Proposition}
\newtheorem{REMA}{Remark}
\newtheorem{EXAMP}{Example}

\newcommand{\theo}{\begin{THEO}}
\newcommand{\algo}{\begin{ALGo} \rm}
\newcommand{\cond}{\begin{COND} \rm}
\newcommand{\conj}{\begin{CONJ}}
\newcommand{\coro}{\begin{CORO}}
\newcommand{\defi}{\begin{DEFI} \rm}
\newcommand{\examp}{\begin{EXAMP} \rm}
\newcommand{\fact}{\begin{FACT}}
\newcommand{\hypo}{\begin{HYPO} \rm}
\newcommand{\lemm}{\begin{LEMM}}
\newcommand{\prob}{\begin{PROB} \rm}
\newcommand{\prop}{\begin{PROP}}
\newcommand{\rema}{\begin{REMA} \rm}
\newcommand{\etheo}{\end{THEO}}
\newcommand{\ealgo}{\end{ALGo}}
\newcommand{\econd}{\end{COND}}
\newcommand{\econj}{\end{CONJ}}
\newcommand{\ecoro}{\end{CORO}}
\newcommand{\edefi}{\end{DEFI}}
\newcommand{\eexamp}{\end{EXAMP} }
\newcommand{\efact}{\end{FACT}}
\newcommand{\ehypo}{\end{HYPO}}
\newcommand{\elemm}{\end{LEMM}}
\newcommand{\eprob}{\end{PROB}}
\newcommand{\eprop}{\end{PROP}}
\newcommand{\erema}{\end{REMA}}
\def\0{\mbox{\bf 0}}
\def\1{\mbox{\bf 1}}
\def\2{\mbox{\bf 2}}
\def\3{\mbox{\bf 3}}
\def\4{\mbox{\bf 4}}
\def\5{\mbox{\bf 5}}
\def\6{\mbox{\bf 6}}
\def\7{\mbox{\bf 7}}
\def\8{\mbox{\bf 8}}
\def\9{\mbox{\bf 9}}

\def\e{{\bm e}}

\def\A{{\bm A}}

\def\SC{\mathcal{S}}

\def\A{\mathcal{A}}

\def\R{\mathbb{R}}

\begin{document}
\title{A successive difference-of-convex approximation method for a class of nonconvex nonsmooth optimization problems
\thanks{Ting Kei Pong is supported in part by Hong Kong Research Grants Council PolyU153085/16p. Akiko Takeda is supported by Grant-in-Aid for Scientific Research (C), 15K00031.}
}
\titlerunning{Successive DC Approximation for Nonconvex and Nonsmooth Problems}
\author{Tianxiang Liu \and Ting Kei Pong \and Akiko Takeda}
\institute{
           T. Liu \at
           Department of Applied Mathematics, The Hong Kong Polytechnic University, Hong Kong. \\
           \email{tiskyliu@polyu.edu.hk}
           \and
           T. K. Pong \at
           Department of Applied Mathematics, The Hong Kong Polytechnic University, Hong Kong.\\
           \email{tk.pong@polyu.edu.hk}
           \and
           A. Takeda \at
           Department of Creative Informatics, Graduate School of Information Science and Technology, the University of Tokyo, Tokyo, Japan.\\
           \email{takeda@mist.i.u-tokyo.ac.jp}
           \at RIKEN Center for Advanced Intelligence Project, 1-4-1, Nihonbashi, Chuo-ku, Tokyo 103-0027, Japan.\\
           \email{akiko.takeda@riken.jp}
}
\authorrunning{T. Liu, T. K. Pong and A. Takeda}
\date{Received: date / Accepted: date}

\maketitle

\begin{abstract}
We consider a class of nonconvex nonsmooth optimization problems whose objective is the sum of a smooth function and a finite number of nonnegative proper closed possibly nonsmooth functions (whose proximal mappings are easy to compute), some of which are further composed with linear maps. This kind of problems arises naturally in various applications when different regularizers are introduced for inducing simultaneous structures in the solutions. Solving these problems, however, can be challenging because of the coupled nonsmooth functions: the corresponding proximal mapping can be hard to compute so that standard first-order methods such as the proximal gradient algorithm cannot be applied efficiently.
In this paper, we propose a successive difference-of-convex approximation method for solving this kind of problems. In this algorithm, we approximate the nonsmooth functions by their Moreau envelopes in each iteration. Making use of the simple observation that Moreau envelopes of nonnegative proper closed functions are continuous {\em difference-of-convex} functions, we can then approximately minimize the approximation function by first-order methods with suitable majorization techniques. These first-order methods can be implemented efficiently thanks to the fact that the proximal mapping of {\em each} nonsmooth function is easy to compute. Under suitable assumptions, we prove that the sequence generated by our method is bounded and any accumulation point is a stationary point of the objective. We also discuss how our method can be applied to concrete applications such as nonconvex fused regularized optimization problems and simultaneously structured matrix optimization problems, and illustrate the performance numerically for these two specific applications.
\end{abstract}

\section{Introduction}\label{intro}
In this paper, we consider the following possibly nonconvex nonsmooth optimization problem:
\begin{equation}
\begin{array}{ll}
\underset{\bm{x}}{\mbox{minimize}} & \displaystyle{F(\bm x) := f(\bm{x}) + P_0(\bm x) + \sum_{i=1}^m P_i(\A_i\bm x),}
\end{array}
\label{sparseprob}
\end{equation}
with the objective satisfying the following assumptions (see the next section for notation and definitions):
\begin{enumerate}[\bf { A}1.]
  \item $f:\R^{n}\rightarrow\R$ is an $L$-smooth function i.e., there exists a constant $L>0$ so that
\begin{align*}
  \|\nabla f(\bm{x})-\nabla f(\bm{v})\|\le L\|\bm{x} -\bm{v}\|
\end{align*}
for any $\bm{x}$, $\bm{v}\in\R^{n}$.
  \item $\A_i:\R^n\rightarrow\R^{n_i}$, $i=1,\ldots,m$, are linear mappings and $P_i:\R^{n_i}\rightarrow\R_+\cup\{\infty\}$, $i=0,\ldots,m$, are proper closed functions. The functions $P_i$, $i=0,\ldots,m$, are continuous in their respective domains, and
  \[
 {\rm dom}\,P_0\cap \bigcap_{i=1}^m\A_i^{-1}({\rm dom}\,P_i) \neq \emptyset.
  \]
  Moreover, the proximal mapping of $\gamma P_i$ is easy to compute for every $\gamma > 0$ and for each $i=0,\ldots,m$. The sets ${\rm dom}\,P_i$, $i=1,\ldots,m$, are closed.
  \item The function $f + P_0$ is level-bounded, i.e., for each $r\in \R$, the set $\{\bm x\in\R^n: f(\bm x) + P_0(\bm x) \le r\}$ is bounded.
\end{enumerate}

Problem~\eqref{sparseprob} arises in many contemporary applications such as structured low rank matrix recovery problems (see, for example, \cite{Markovsky2008}), nonconvex fused regularized optimization problems (see, for example, \cite{ParakhSelesnick15} and Example~\ref{examp:fusedreg} in Section~\ref{sec:struct}) and simultaneously structured matrix optimization problems (see, for example, \cite{richard2012} and Example~\ref{examp:lr_sparse} in Section~\ref{sec:struct}). In these applications, the $P_i$'s are used for inducing desirable structures in the solutions and they are typically functions whose proximal mappings are easy to compute. If only one such function appears in \eqref{sparseprob}, i.e., $m = 0$, then some standard first-order methods such as the proximal gradient algorithm or its variants can be applied to solving \eqref{sparseprob} efficiently, because these algorithms only require the computation of $\nabla f$ and the proximal mapping of $\gamma P_0$ ($\gamma > 0$) in each iteration. However, in all the aforementioned applications, there are always more than one such structure-inducing functions in \eqref{sparseprob} (i.e., $m\ge 1$) and the ${\cal A}_i$'s might not always be identity mappings. Then the proximal gradient algorithm and its variants cannot be applied efficiently, because the proximal mapping of $\bm x\mapsto P_0(\bm x) + \sum_{i=1}^m P_i(\A_i\bm x)$ can be hard to compute in general.

When the function $f$ and the $P_i$'s are all convex functions, one alternative approach for solving \eqref{sparseprob} is the alternating direction method of multipliers (ADMM); see, for example, \cite{EckBer92,GabMer76}. This method can be applied to \eqref{sparseprob} by suitably introducing slack variables that transform the problem into a linearly constrained problem, and each iteration only requires computing the proximal mappings of $f$ and $\gamma P_i$'s, as well as an update of an auxiliary (dual) variable. However, it is known that the ADMM does not necessarily converge if the $P_i$'s are nonconvex and $m\ge 1$; see, for example, \cite[Example~7]{LiPong15}.
In the case when $P_i$'s are nonconvex but globally Lipschitz for $i=0,\ldots,m$ , and ${\cal A}_i$ is the identity mapping for all $i$, a new method for solving \eqref{sparseprob} was introduced in a series of work \cite{Yu2013,Yu2015}. Their method is based on the so-called proximal average of $P_i$'s, and each iteration involves only the computations of $\nabla f$ and the proximal mappings of $\gamma P_i$'s. However, it was only shown that any accumulation point of the sequence generated by their method is a stationary point of a certain smooth approximation of \eqref{sparseprob}. Moreover, their method was designed for the case when $P_i$'s are globally Lipschitz, and the convergence behavior of their method is unknown when some non-Lipschitz functions such as the $\ell_p$ quasi-norm or the indicator function of some closed sets (such as the set of all $k$-sparse vectors) are present in \eqref{sparseprob}.

In this paper, we propose a new method for solving \eqref{sparseprob} that is ready to take advantage of the ease of proximal mapping computations and has convergence guarantee under suitable assumptions, without imposing convexity nor globally Lipschitz continuity on $P_i$'s. We call our method the successive difference-of-convex approximation method (SDCAM). In this method, we construct an approximation to the objective of \eqref{sparseprob} in each iteration using the Moreau envelopes of the $\lambda_{i,t}P_i$, $i=1,\ldots,m$, where $t$ is the number of iteration and $\{\lambda_{i,t}\}$ are nonincreasing positive sequences satisfying $\lim_{t\to\infty}\lambda_{i,t} = 0$; a suitable approximate stationary point of this approximation function is then taken to be the next iterate $\bm x^{t+1}$ of our algorithm. The point $\bm x^{t+1}$ can be found efficiently by recalling that the Moreau envelopes involved, despite being nonsmooth in general due to the possible nonconvexity of the $P_i$'s, are continuous difference-of-convex functions. Thus, one can incorporate majorization techniques in some standard first-order methods such as the proximal gradient algorithm for finding $\bm x^{t+1}$ in each iteration. Moreover, when such first-order methods are applied, the main computational cost per inner iteration typically only depends on the computations of $\nabla f$ and the proximal mappings of $\gamma P_i$, $i=0,\ldots,m$, $\gamma > 0$, which are inexpensive in many applications. This suggests that the SDCAM can be applied efficiently for solving \eqref{sparseprob}. More details of this algorithm will be discussed in Section~\ref{sec3}, where we also prove that the sequence $\{\bm x^t\}$ generated is bounded and any accumulation point is a stationary point of \eqref{sparseprob} under suitable assumptions.

The rest of the paper is organized as follows. In Section~\ref{sec2}, we introduce notation and some preliminary results. Our SDCAM is presented and its convergence is analyzed under suitable assumptions in Section~\ref{sec3}. We then discuss how our method can be applied to various kinds of structured optimization problems including some nonconvex fused regularized optimization problems, some simultaneously sparse and low rank matrix optimization problems, and the low rank nearest correlation matrix problem, in Section~\ref{sec:struct}. We also perform numerical experiments on some of these applications to demonstrate the efficiency of our algorithm in Section~\ref{sec5}. Finally, we present some concluding remarks in Section~\ref{sec6}.

\section{Notation and preliminaries}\label{sec2}

In this paper, vectors and matrices are represented in bold lower case letters and upper case letters, respectively.
The inner product of two vectors $\bm{a}$ and $\bm{b}\in\mathbb{R}^{n}$ are denoted by $\bm{a}^\top\bm{b}$ or $\bm{b}^\top\bm{a}$, and we use $\|\bm a\|_0$, $\|\bm a\|_1$ and $\|\bm a\|$ to denote the number of nonzero entries, the $\ell_1$ norm and the $\ell_2$ norm of $\bm a$, respectively. Moreover, we use $\Diag(\bm a)$ to denote the diagonal matrix whose diagonal is $\bm a$.
For two matrices $\bm{A}$ and $\bm{B}\in\R^{m{\times}n}$, their Hadamard (entrywise) product is denoted by $\bm A\circ\bm B$. We also use $\|\bm A\|_*$ and $\|\bm A\|_F$ to denote the nuclear norm and the Fr\"{o}benius norm of $\bm A$, respectively, and let $\vect(\bm A)\in \R^{mn{\times}1}$ denote the vectorization of $\bm A$, which is
obtained by stacking the columns of $\bm A$ on top of one another. Furthermore, we use $\sigma_{\max}(\bm A)$ to denote the largest singular value of $\bm A$. The space of symmetric $n\times n$ matrices is denoted by $\SC^n$. For a matrix $\bm X\in \SC^n$, we use $\diag(\bm X)\in \R^n$ to denote its diagonal and $\lambda_{\max}(\bm X)$ to denote its largest eigenvalue. We write $\bm X\succeq 0$ if $\bm X$ is positive semidefinite. For a linear operator $\A$, we let $\A^*$ denote its adjoint.

A function $h:\R^{n}\to \R\cup \{\infty\}$ is said to be proper if ${\rm dom}\,h:=\{\bm x:\; h(\bm x)<\infty\}\neq \emptyset$. Such a function is said to be closed if it is lower semicontinuous. Following \cite[Definition~8.3]{RockWets98}, for a proper function $h$, the limiting and horizon subdifferentials at $\bm x\in {\rm dom}\,h$ are defined respectively as
\[
\begin{split}
&\partial h(\bm x)= \left\{\bm u:\; \exists \bm u^t \to \bm u, \bm x^t \stackrel{h}{\to}\bm x \ \ \mbox{with}\ \bm u^t\in \hat\partial h(\bm x^t)\ \mbox{for each $t$}\right\},\\
&\partial^\infty h(\bm x)= \left\{\bm u:\; \exists \alpha_t\downarrow 0, \alpha_t\bm u^t \to \bm u, \bm x^t \stackrel{h}{\to}\bm x \ \ \mbox{with}\ \bm u^t\in \hat\partial h(\bm x^t)\ \mbox{for each $t$}\right\},
\end{split}
\]
where $\hat \partial h(\bm w) := \left\{\bm u:\; \liminf\limits_{\bm y\to\bm w,\bm y\neq \bm w}\frac{h(\bm y)-h(\bm w)-{\bm u}^\top(\bm y-\bm w)}{\|\bm y-\bm w\|}\ge 0\right\}$, and the notation $\bm x^t \stackrel{h}{\to}\bm x$ means $\bm x^t \to \bm x$ and $h(\bm x^t)\to h(\bm x)$. We also define $\partial h(\bm x) = \partial^\infty h(\bm x) := \emptyset$ when $\bm x\notin {\rm dom}\,h$. It is easy to show that at any $\bm x\in {\rm dom}\,h$, the limiting and horizon subdifferentials have the following robustness property:
\begin{equation}\label{robust}
\begin{split}
  \{\bm u:\; \exists \bm u^t \to \bm u, \bm x^t \stackrel{h}{\to}\bm x \ \mbox{with}\ \bm u^t \in \partial h(\bm x^t) \ \mbox{for each }t\}&\subseteq \partial h(\bm x),\\
  \{\bm u:\; \exists \alpha_t \downarrow 0, \alpha_t\bm u^t \to \bm u, \bm x^t \stackrel{h}{\to}\bm x \ \mbox{with}\ \bm u^t \in \partial h(\bm x^t) \ \mbox{for each }t\}&\subseteq \partial^\infty h(\bm x).
\end{split}
\end{equation}
The limiting subdifferential at $\bm x$ reduces to $\{\nabla h(\bm x)\}$ if $h$ is continuously differentiable at $\bm x$ \cite[Exercise~8.8(b)]{RockWets98}, and reduces to the convex subdifferential if $h$ is proper convex \cite[Proposition~8.12]{RockWets98}.

For a proper closed function $h$ with $\inf h > -\infty$, we will also need its Moreau envelope for any given $\lambda > 0$, which is defined as
\[
e_\lambda h(\bm x) := \inf_{\bm y}\left\{\frac1{2\lambda}\|\bm x - \bm y\|^2 + h(\bm y)\right\}.
\]
This function is finite everywhere \cite[Theorem~1.25]{RockWets98}.
It is not hard to see that
\begin{equation}\label{inf_ineq}
e_\lambda h(\bm x)\le h(\bm x)
\end{equation}
for all $\bm x$.
The infimum in the definition of Moreau envelope is attained at the so-called proximal mapping of $\lambda h$ at $\bm x$, which is defined as
\begin{align*}
\prox_{\lambda h}(\bm{x}):=\Argmin_{\bm{u}\in \R^{n}}\left\{ \frac{1}{2\lambda}\|\bm{x}-\bm{u}\|^2 + h(\bm{u})\right\}.
\end{align*}
This set is always nonempty because $h$ is proper closed and bounded below \cite[Theorem~1.25]{RockWets98}. Let $\bm \zeta_\lambda\in \prox_{\lambda h}(\bm{x})$. Then we have from \cite[Theorem~10.1]{RockWets98} and \cite[Exercise~8.8(c)]{RockWets98} that
\begin{equation}\label{key_inclusion}
\frac{1}{\lambda}(\bm x - \bm \zeta_\lambda) \in \partial h(\bm \zeta_\lambda).
\end{equation}
Furthermore, we have the following simple lemma, which should be well known. We provide a short proof for self-containedness.
\begin{LEMM}\label{prox_limit}
  Let $h$ be a proper closed function with $\inf h > -\infty$ and let $\bm x^*\in {\rm dom}\,h$. Suppose that $\bm x^t\to \bm x^*$, $\lambda_t \downarrow 0$ and pick any $\bm \zeta^t\in \prox_{\lambda_t h}(\bm x^t)$ for each $t$. Then it holds that $\bm \zeta^t\in {\rm dom}\,h$ for all $t$ and $\bm \zeta^t \to \bm x^*$.
\end{LEMM}
\begin{proof}
Under the assumptions, we have the following inequality:
\[
\frac1{2\lambda_t}\|\bm x^t - \bm \zeta^t\|^2 + \inf h \le \frac1{2\lambda_t}\|\bm x^t - \bm \zeta^t\|^2 + h(\bm \zeta^t) = e_{\lambda_t} h(\bm x^t)\le \frac1{2\lambda_t}\|\bm x^t - \bm x^*\|^2 + h(\bm x^*).
\]
Hence, we have $\bm \zeta^t\in {\rm dom}\,h$ for all $t$ and
\[
\|\bm \zeta^t - \bm x^*\|\le \|\bm \zeta^t - \bm x^t\| + \|\bm x^t -\bm x^*\|\le \sqrt{2\lambda_t(h(\bm x^*) - \inf h)+\|\bm x^t - \bm x^*\|^2} + \|\bm x^t - \bm x^*\| \to 0.
\]
\qedwhite
\end{proof}

Finally, recall that for a nonempty closed set $C$, the indicator function is defined as
\[
\delta_C(\bm x) = \begin{cases}
  0 & {\rm if}\  \bm x\in C,\\
  \infty & {\rm else}.
\end{cases}
\]
We define the (limiting) normal cone at any $\bm x\in C$ as $N_C(\bm x):= \partial \delta_C(\bm x)$. We let ${\rm dist}(\bm x, C) := \inf_{\bm y\in C}\|\bm y - \bm x\|$. The set of points in the nonempty closed set $C$ that are closest to a given $\bm x$ is denoted by $\proj_C(\bm x)$. One can observe that $\proj_C = \prox_{\delta_C}$. The set $\proj_C(\bm x)$ at a given $\bm x$ is always nonempty for a nonempty closed set $C$, and is a singleton when $C$ is in addition convex.

\section{Solution method for nonconvex nonsmooth optimization problems}\label{sec3}
\subsection{Successive difference-of-convex approximation method }
In this paper, we consider problem~\eqref{sparseprob} and assume that its objective satisfies the assumptions {\bf A1}, {\bf A2} and {\bf A3} in Section~\ref{intro}.
%the following assumptions:
%\begin{enumerate}[\bf { A}1.]
%  \item $f:\R^{n}\rightarrow\R_+$ is an $L$-smooth function i.e., there exists a constant $L>0$ so that
%\begin{align*}
%  \|\nabla f(\bm{x})-\nabla f(\bm{v})\|\le L\|\bm{x} -\bm{v}\|
%\end{align*}
%for any $\bm{x}$, $\bm{v}\in\R^{n}$.
%  \item $\A_i:\R^n\rightarrow\R^{n_i}$, $i=1,\ldots,m$, are linear mappings and $P_i:\R^{n_i}\rightarrow\R_+\cup\{\infty\}$, $i=0,\ldots,m$, are proper closed functions. The functions $P_i$, $i=0,\ldots,m$, are continuous in their respective domains, and
%  \[
% {\rm dom}\,P_0\cap \bigcap_{i=1}^m\A_i^{-1}({\rm dom}\,P_i) \neq \emptyset.
%  \]
%  Moreover, the proximal mapping of $\gamma P_i$ is easy to compute for every $\gamma > 0$ and for each $i=0,\ldots,m$. The sets ${\rm dom}\,P_i$, $i=1,\ldots,m$, are closed.
%  \item The function $f + P_0$ is level-bounded.
%\end{enumerate}
%Problem \eqref{sparseprob} arises naturally in various structured optimization problems. The functions $P_i$ can be used to model or induce different structures, and are nonconvex nonsmooth in general.
We will discuss some concrete applications of \eqref{sparseprob} in more details in Section~\ref{sec:struct}. In this section, we present an algorithm for solving \eqref{sparseprob}.

Notice that \eqref{sparseprob} is in general a nonsmooth nonconvex optimization problem. The nonsmooth nonconvex function $P_0 + \sum_{i=1}^m P_i\circ\A_i$ can be complicated in practice and handling it directly can be challenging. Indeed, although the proximal mappings of $\gamma P_i$, $i=0,\ldots,m$, are easy to compute, the proximal mapping of $P_0 + \sum_{i=1}^m P_i\circ \A_i$ may be hard to evaluate and hence the classical proximal gradient algorithm and its variants cannot be adapted directly and efficiently for solving \eqref{sparseprob}. In this paper, we suitably adapt a ``smoothing" scheme for solving the above nonconvex nonsmooth problem.
In this approach, in each iteration, we minimize the auxiliary function
\begin{equation}\label{sparseprob2}
F_{\bm \lambda}(\bm x) := f(\bm x) + P_0(\bm x) + \sum_{i=1}^m e_{\lambda_i} P_i(\A_i\bm x)
\end{equation}
approximately and then update $\bm x$ and ${\bm \lambda} = (\lambda_1,\cdots,\lambda_m)$, where $e_{\lambda_i}P_i$ is the Moreau envelope of $P_i$.

When $P_i$, $i=1,\ldots,m$ are all {\em convex functions}, the corresponding functions $e_{\lambda_i}P_i$ are Lipschitz differentiable \cite[Proposition~12.29]{BauCom11}. Hence, the function $F_{\bm \lambda}$ becomes the sum of a nonsmooth function $P_0$ and a smooth function, and can be minimized efficiently using, for example, the proximal gradient algorithm and its variants. This smoothing strategy has been widely used in the literature for convex problems; see \cite{Nes05}, and also \cite{BCG11} for a software package for convex optimization problems based on smoothing techniques. However, in our setting, $P_i$ is {\em not necessarily convex}. Thus, the corresponding Moreau envelope $e_{\lambda_i}P_i$ is {\it not necessarily smooth} and it is unclear whether $F_{\bm \lambda}$ can be minimized efficiently at first glance.

The key ingredient in our approach (where $P_i$ is possibly nonconvex) is the simple observation that for any nonnegative proper closed function $P$ and any $\mu > 0$,
\begin{equation}\label{moreauEnv}
  e_\mu P(\bm u) = \frac1{2\mu} \|\bm u\|^2 - \underbrace{\sup_{\bm y\in {\rm dom}\, P}\left\{\frac1\mu\bm u^\top\bm y - \frac1{2\mu}\|\bm y\|^2 - P(\bm y)\right\}}_{D_{\mu,P}(\bm u)}.
\end{equation}
Such a decomposition has been noted in \cite{Asplund73} when $P = \delta_C$ for some nonempty closed set $C$, and in \cite[Proposition~3]{Lucet06} for the general case. Then $D_{\mu,P}$, as the supreme of affine functions and being finite-valued, is convex continuous. Moreover, using the definition of $e_\mu P(\bm u)$, $\prox_{\mu P}(\bm u)$ and (6), we see that the supremum in $D_{\mu,P}(\bm u)$ is attained at any point in $\prox_{\mu P}(\bm u)$. Let $\bm y^* \in \prox_{\mu P}(\A\bm x)$. Then $\bm y^* \in {\rm dom}\,P$ and we have for any $\bm w$ that
\[
\begin{split}
 & D_{\mu, P}(\bm w) - D_{\mu, P}({\cal A}\bm x) \\
 &= \sup_{\bm y\in {\rm dom}\, P}\left\{\frac1\mu\bm w^\top\bm y - \frac1{2\mu}\|\bm y\|^2 - P(\bm y)\right\} - \sup_{\bm y\in {\rm dom}\, P}\left\{\frac1\mu\bm (\A\bm x)^\top\bm y - \frac1{2\mu}\|\bm y\|^2 - P(\bm y)\right\} \\
 &\ge \frac1\mu\bm w^\top\bm y^* - \frac1{2\mu}\|\bm y^*\|^2 - P(\bm y^*) - \left(\frac1\mu\bm (\A\bm x)^\top\bm y^* - \frac1{2\mu}\|\bm y^*\|^2 - P(\bm y^*)\right) \\
 &= \frac1\mu {\bm y^*}^\top(\bm w - \A \bm x).
\end{split}
\]
This implies $\frac1\mu\prox_{\mu P}(\A\bm x) \subseteq \partial D_{\mu,P}(\A\bm x)$, from which we deduce further that
\begin{align}\label{inclusion}
  \frac1\mu \A^*\prox_{\mu P}(\A\bm x)\subseteq \A^*\partial D_{\mu,P}(\A\bm x) = \partial (D_{\mu,P}\circ\A)(\bm x),
\end{align}
where the last equality follows from \cite[Theorem~23.9]{Rock70} because $D_{\mu,P}$ is convex continuous. Thus, \eqref{sparseprob2} is the sum of a smooth function $f$, a nonsmooth nonconvex function $P_0$ whose proximal mapping is easy to compute, and a continuous difference-of-convex function such that a subgradient corresponding to its concave part is easy to compute; thanks to \eqref{inclusion} and Assumption {\bf A2}. Proximal gradient methods with majorization techniques can then be suitably applied to minimizing \eqref{sparseprob2}. For instance, one can apply the NPG$_{\rm major}$ described in the appendix. Specifically, one can apply NPG$_{\rm major}$ with
\[
h(\bm x) = f(\bm x) + \sum_{i=1}^m\frac1{2\lambda_i}\|\A_i\bm x\|^2, \ P(\bm x) = P_0(\bm x),\ g(\bm x) = \sum_{i=1}^m D_{\lambda_i,P_i}(\A_i\bm x).
\]
It is routine to check that this choice of $h$, $P$ and $g$ satisfies the assumptions required in the appendix. Moreover, the $F_{\bm \lambda}$ is level-bounded because $f+P_0$ is level-bounded by assumption and $e_{\lambda_i}P_i$ are nonnegative for each $i=1,\ldots,m$ since $P_i$ are nonnegative. Finally, $F_{\bm \lambda}$ is continuous in its domain because $P_0$ is. Hence all assumptions required in the appendix for applying NPG$_{\rm major}$ are satisfied and the method can be applied to minimizing $F_{\bm \lambda}$ by initializing at {\em any} point $\bm x^0\in {\rm dom}\,P_0$.

We now describe our method for solving \eqref{sparseprob} with its update rules below in Algorithm \ref{alg_sdcam}. We call this method the successive difference-of-convex approximation method (SDCAM).
\begin{algorithm}
\caption{The SDCAM for \eqref{sparseprob}}\label{alg_sdcam}

\begin{description}
\item[\bf Step 0.] Pick $m+1$ sequences of positive numbers with $\epsilon_t\downarrow 0$ and $\lambda_{i,t} \downarrow 0$ for $i=1,\ldots,m$, an $\bm x^{\rm feas} \in {\rm dom}\,P_0\cap \bigcap_{i=1}^m \A_i^{-1}({\rm dom}\,P_i)$, and an $\bm x^0\in {\rm dom}\,P_0$. Set $t = 0$.

\item[\bf Step 1.]  If $F_{\bm \lambda_t}(\bm x^t) \le F_{\bm \lambda_t}(\bm x^{\rm feas})$, set $\bm x^{t,0} = \bm x^t$. Else, set $\bm x^{t,0} = \bm x^{\rm feas}$.
\item[\bf Step 2.] Approximately minimize
  $F_{\bm \lambda_t}(\bm x)$, starting at $\bm x^{t,0}$, and terminating at $\bm x^{t,l_t}$ when
\begin{equation}\label{condition}
\begin{split}
  &{\rm dist}\left(0,\nabla f(\bm x^{t,l_t}) + \partial P_0(\bm x^{t,l_t+1}) + \sum_{i=1}^m\frac{1}{\lambda_{i,t}}\A^*_i[\A_i\bm x^{t,l_t} - \prox_{\lambda_{i,t}P_i}(\A_i\bm x^{t,l_t})]\right) \le \epsilon_t, \\
  &{\rm and}\ \ \ \ \|\bm x^{t,l_t+1} - \bm x^{t,l_t}\|\le \epsilon_t,\ \ \ \ F_{\bm \lambda_t}(\bm x^{t,l_t})\le F_{\bm \lambda_t}(\bm x^{t,0}).
\end{split}
\end{equation}
\item[\bf Step 3.] Update $\bm x^{t+1} = \bm x^{t,l_t}$ and $t=t+1$. Go to \bf Step 1.
\end{description}
\end{algorithm}

We would like to point out that {\bf Step 1} in SDCAM is crucial in our convergence analysis: this strategy was also used in the penalty decomposition method in \cite{LuZh13}. As we shall see in the proof of Theorem~\ref{optimalitycond} below, it ensures that \eqref{robust} can be applied at an accumulation point of $\{\bm x^t\}$.

\subsection{Theoretical guarantee for global convergence}

In this section, we first discuss how $F_{\bm \lambda_t}$ can be approximately minimized so that \eqref{condition} is satisfied at the $t$-th iteration and comment on the computational complexity. Then we prove the convergence of the SDCAM under suitable assumptions.

As discussed above, $F_{\bm \lambda_t}$ can be minimized by the NPG$_{\rm major}$ outlined in the appendix. Moreover, due to \eqref{inclusion}, one can choose $\bm \zeta^{t,l} = \sum_{i=1}^m\frac{1}{\lambda_{i,t}}\A_i^*\bm \zeta_i^{t,l}$ in the algorithm with
\begin{equation}\label{choicezeta}
\bm \zeta_i^{t,l}\in \prox_{\lambda_{i,t}P_i}(\A_i\bm x^{t,l})
\end{equation}
for each $i=1,\ldots,m$ and $l\ge 0$ so that $\sum_{i=1}^m\frac{1}{\lambda_{i,t}}\A_i^*\bm \zeta_i^{t,l}$ lies in the subdifferential of $\sum_{i=1}^m (D_{\lambda_{i,t},P_i}\circ \A_i)$ at $\bm x^{t,l}$.
Using this special version of NPG$_{\rm major}$, we can show that the termination criterion \eqref{condition} is satisfied after finitely many inner iterations.

\begin{THEO}\label{thm0}
  Suppose that the NPG$_{\rm major}$ is applied with $\bm \zeta^{t,l} = \sum_{i=1}^m\frac{1}{\lambda_{i,t}}\A_i^*\bm \zeta_i^{t,l}$, where $\bm \zeta_i^{t,l}$ are chosen as in \eqref{choicezeta}, to minimizing $F_{\bm \lambda_t}$ in the $t$-th iteration of SDCAM. Then the criterion \eqref{condition} is satisfied after finitely many inner iterations.
\end{THEO}
\begin{proof}
According to the convergence properties of the NPG$_{\rm major}$, one obtains a sequence $\{\bm x^{t,l}\}_{l\ge 0}$ satisfying
\begin{enumerate}
  \item $\lim\limits_{l\to\infty}\|\bm x^{t,l+1}-\bm x^{t,l}\| = 0$ (Proposition~\ref{prop:conv} in the appendix), $F_{\bm \lambda_t}(\bm x^{t,l})\le F_{\bm \lambda_t}(\bm x^{t,0})$ (thanks to \eqref{descent}); and
  \item for any $l\ge 0$ (see \eqref{prox-subprob}),
  \begin{equation}\label{conditionsub}
  \bm x^{t,l+1}\!\in\! \Argmin_{\bm x}\left\{\left(\nabla f(\bm x^{t,l})+\sum_{i=1}^m\frac{\bm \omega^{t,l}_i}{\lambda_{i,t}}\right)^\top\!\! \bm x + \frac{\bar L_{t,l}}{2}\|\bm x - \bm x^{t,l}\|^2 + P_0(\bm x)\right\},
  \end{equation}
  where $\bm \omega^{t,l}_i:= \A_i^*[\A_i\bm x^{t,l}- \bm \zeta_i^{t,l}]$.
  Here, the sequence $\{\bar L_{t,l}\}_{l\ge 0}$ can be shown to be bounded; see Proposition~\ref{prop:bd} in the appendix.
\end{enumerate}
Using \cite[Exercise~8.8(c)]{RockWets98}, the condition \eqref{conditionsub} implies
\[
\begin{split}
  &\bm 0 \in \nabla f(\bm x^{t,l})+ \sum_{i=1}^m\frac{1}{\lambda_{i,t}}\A_i^*[\A_i\bm x^{t,l}- \bm \zeta_i^{t,l}] +\bar L_{t,l}(\bm x^{t,l+1} - \bm x^{t,l}) + \partial P_0(\bm x^{t,l+1}),\\
  \Rightarrow & -\bar L_{t,l}(\bm x^{t,l+1} - \bm x^{t,l}) \in \nabla f(\bm x^{t,l})+ \sum_{i=1}^m\frac{1}{\lambda_{i,t}}\A_i^*[\A_i\bm x^{t,l}- \bm \zeta_i^{t,l}] + \partial P_0(\bm x^{t,l+1}),
\end{split}
\]
from which \eqref{condition} can be seen to hold with $l_t = l$ when $l$ is sufficiently large because $\lim\limits_{l\to\infty}\|\bm x^{t,l+1}-\bm x^{t,l}\| = 0$ and $\{\bar L_{t,l}\}_{l\ge 0}$ is bounded.\qedwhite
\end{proof}

\begin{REMA}[{{\bf Computational complexity}}]
  \normalfont Suppose that the NPG$_{\rm major}$ is applied to minimizing $F_{\bm \lambda_t}$ in each iteration of SDCAM, with the $\bm \zeta^{t,l}$ chosen as in Theorem~\ref{thm0}. Then one has to repeatedly solve subproblems of the form \eqref{conditionsub} for various values of $\bm \lambda_t$ and $\beta > 0$ (in place of $\bar L_{t,l}$).
These computations are easy under the assumption that the proximal mapping $\gamma P_i$, $i=1,\ldots,m$, $\gamma > 0$, is easy to compute.
Indeed, the subproblems can be rewritten as
  \begin{equation}
    \bm x^{t,l+1}\in \prox_{\frac{1}{\beta}P_0}\left(\bm x^{t,l}- \frac{1}{\beta} \left(\nabla f(\bm x^{t,l})+ \sum_{i=1}^m\frac{1}{\lambda_{i,t}}\A_i^*[\A_i\bm x^{t,l}- \bm \zeta_i^{t,l}]\right)\right),
\label{subproblem_prox}
  \end{equation}
where  $\bm \zeta_i^{t,l}\in \prox_{\lambda_{i,t}P_i}(\A_i\bm x^{t,l})$.
\end{REMA}

We now state and prove our convergence result for SDCAM. We will comment on \eqref{CQ} in Remark~\ref{rem2} below before proving the theorem.
\begin{THEO}[{{\bf Convergence of SDCAM}}]
  Let $\{\bm x^t\}$ be the sequence generated by SDCAM for solving \eqref{sparseprob}. Then $\{\bm x^t\}$ is bounded. Let $\bm x^*$ be an accumulation point of this sequence. Then we have the following results.
  \begin{enumerate}[{\rm (i)}]
    \item It holds that $\bm x^*\in {\rm dom}\,P_0 \cap \bigcap_{i=1}^m \A_i^{-1}({\rm dom}\,P_i)$.
    \item Suppose that
    \begin{equation}\label{CQ}
    \begin{split}
    &\bm y_0 + \sum_{i=1}^m \A_i^*\bm y_i = \bm 0 \ \ \mbox{\rm and}\ \ \bm y_0 \in \partial^\infty P_0(\bm x^*),\ \ \bm y_i \in \partial^\infty P_i(\A_i\bm x^*) \ \mbox{\rm for }i=1,\ldots,m \\
    &\qquad\qquad\qquad\qquad \Longrightarrow \ \bm y_i = \bm 0\ \ \mbox{\rm for }i=0,\ldots,m.
    \end{split}
    \end{equation}
    Then $\bm x^*$ is a stationary point of \eqref{sparseprob}, i.e.,
    \begin{equation}\label{opt_con}
     \bm 0 \in \nabla f(\bm x^*) + \partial P_0(\bm x^*) + \sum_{i=1}^m\A_i^*\partial P_i(\A_i\bm x^*).
    \end{equation}
  \end{enumerate}
\label{optimalitycond}
\end{THEO}

\begin{REMA}\label{rem2}{\bf (Comments on condition \eqref{CQ})}
% \normalfont The following remarks are for the claim  (ii) in Theorem \ref{optimalitycond}.
  \begin{enumerate}[{\rm (i)}]
  \item Condition \eqref{CQ} is a classical constraint qualification for nonconvex nonsmooth optimization problems; see \cite[Corollary~10.9]{RockWets98}. It is satisfied, for example, when $\A_i$ equals the identity map for all $i$, and all but one $P_i$ are locally Lipschitz so that $\partial^\infty P_i(\bm x^*) = \{\bm 0\}$ for all but one $P_i$; see \cite[Exercise~10.10]{RockWets98}.
  \item Under \eqref{CQ}, it can be shown using \cite[Theorem~10.1]{RockWets98}, \cite[Proposition~10.5]{RockWets98} and \cite[Theorem~10.6]{RockWets98} that any local minimizer $\bm x^*$ of \eqref{sparseprob} satisfies \eqref{opt_con}.
\end{enumerate}
\end{REMA}

\begin{proof}
  Using the nonnegativity of $P_i$, the last criterion in \eqref{condition} and the definitions of $F_{\bm \lambda}$ and $\bm x^{t,0}$, we see that
  \begin{equation}\label{level_bd}
  f(\bm x^t) + P_0(\bm x^{t})\le F_{\bm \lambda_{t-1}}(\bm x^{t}) \le F_{\bm \lambda_{t-1}}(\bm x^{\rm feas}) \le F(\bm x^{\rm feas})=:F_{\rm feas},
  \end{equation}
  where the last inequality follows from the definitions of $F$, $F_{\bm \lambda}$ and \eqref{inf_ineq}.
  From this, one immediately conclude that $\{\bm x^t\}$ is bounded because $f+P_0$ is level-bounded.

  Next, let $\bm x^*$ be an accumulation point of $\{\bm x^t\}$. Then there exists a subsequence $\{\bm x^t\}_{t\in{\cal I}}$ so that $\lim\limits_{t\in \cal I}\bm x^{t} = \bm x^*$. Using this, \eqref{level_bd}, and the lower semicontinuity of $f+P_0$, we further see that
  \[
  f(\bm x^*) + P_0(\bm x^*) \le \liminf_{t\in \cal I} f(\bm x^t) + P_0(\bm x^{t})\le F_{\rm feas} < \infty.
  \]
  This shows that $\bm x^*\in {\rm dom}\,P_0$. On the other hand, since $P_i$ is nonnegative, we have
  \[
  \begin{split}
    0&\le \frac12{\rm dist}^2(\A_i\bm x,{\rm dom}\,P_i) = \inf_{\bm y\in {\rm dom}\,P_i}\left\{\frac12\|\A_i\bm x - \bm y\|^2\right\}\\
    &\le \inf_{\bm y\in {\rm dom}\,P_i}\left\{\frac12\|\A_i\bm x - \bm y\|^2 + \lambda_{i,t-1}P_i(\bm y)\right\} = \lambda_{i,t-1} e_{\lambda_{i,t-1}}P_i(\A_i\bm x)
  \end{split}
  \]
  for all $\bm x$ and for each $i=1,\ldots,m$. Using this, the finiteness of $\underline{\ell} := \inf\{f + P_0\}$ (thanks to the level-boundedness of $f+P_0$), and the definition of $F_{\bm \lambda}$, we have for each $i = 1,\ldots,m$ that
  \[
  \underline{\ell} + \frac1{2\lambda_{i,t-1}}{\rm dist}^2(\A_i\bm x^t, {\rm dom}\,P_i) \le \underline{\ell} + e_{\lambda_{i,t-1}}P_i(\A_i\bm x^t)\le F_{\bm \lambda_{t-1}}(\bm x^{t})\le F_{\rm feas},
  \]
  where the last inequality follows from \eqref{level_bd}.
  Since $\lambda_{i,t-1}\downarrow 0$, we conclude that ${\rm dist}^2(\A_i\bm x^*,{\rm dom}\,P_i) \le 0$ and hence $\A_i\bm x^*\in {\rm dom}\,P_i$ because ${\rm dom}\,P_i$ is closed.

  We now prove \eqref{opt_con} under \eqref{CQ}. For notational simplicity, let $\bm y^{t+1} := \bm x^{t,l_t+1}$. Then $\lim\limits_{t\in \cal I}\bm y^t = \bm x^*$ thanks to the second relation in \eqref{condition}. Moreover, from the first relation in \eqref{condition}, we see that there exist $\bm \xi^t$ with $\|\bm \xi^t\|\le \epsilon_{t-1}$, $\bm \eta^t \in \partial P_0(\bm y^t)$ and $\bm \zeta_i^t\in \prox_{\lambda_{i,t-1}P_i}(\A_i\bm x^t)$ for each $i=1,\ldots,m$ so that
  \begin{equation}\label{condition2}
    \bm \xi^t = \nabla f(\bm x^t) + \bm \eta^t + \sum_{i=1}^m\frac1{\lambda_{i,t-1}}\A_i^*(\A_i\bm x^t - \bm \zeta_i^t).
  \end{equation}

  Define
    \[
    r_t:= \|\bm \eta^t\| + \sum_{i=1}^m\frac1{\lambda_{i,t-1}}\|\A_i^*(\A_i\bm x^t - \bm \zeta_i^t)\|.
    \]
    We claim that $\{r_t\}_{t\in \cal I}$ is bounded.
    Suppose to the contrary that $\{r_t\}_{t\in \cal I}$ is unbounded and we assume without loss of generality that $\lim\limits_{t\in \cal I}r_t = \infty$ and $\inf\limits_{t\in \cal I} r_t > 0$. Then the sequences $\{\frac{1}{r_t}\bm \eta^t\}_{t\in \cal I}$ and $\{\frac1{\lambda_{i,t-1}r_t}\A_i^*(\A_i\bm x^t - \bm \zeta_i^t)\}_{t\in \cal I}$ for $i=1,\ldots,m$ are bounded. Without loss of generality, we may assume
    \begin{equation}\label{proof1rel3}
    \lim\limits_{t\in \cal I}\frac{\bm \eta^t}{r_t} =\bm \eta^*\ \ {\rm and}\ \  \lim\limits_{t\in \cal I}\A_i^*\left(\frac{\A_i\bm x^t - \bm \zeta_i^t}{\lambda_{i,t-1}r_t}\right)= \bm \chi_i^*
    \end{equation}
    for some $\bm \eta^*$ and $\bm \chi_i^*$, $i=1,\ldots,m$.
    Notice that
    \begin{equation}\label{proof1rel2}
    1 = \frac{ \|\bm \eta^t\| + \sum_{i=1}^m\frac1{\lambda_{i,t-1}}\|\A_i^*(\A_i\bm x^t - \bm \zeta_i^t)\|}{r_t} \ \Rightarrow 1 = \|\bm \eta^*\| + \sum_{i=1}^m\|\bm \chi_i^*\|.
    \end{equation}
    In addition, by dividing $r_t$ from both sides of \eqref{condition2} and passing to the limit along $t\in \cal I$, we conclude that
    \begin{equation}\label{sum}
    \bm 0 = \bm \eta^* + \sum_{i=1}^m\bm \chi_i^*.
    \end{equation}
    On the other hand, since $\bm \eta^t \in \partial P_0(\bm y^t)$ and $\lim\limits_{t\in \cal I}r_t = \infty$, we have from \eqref{proof1rel3}, the continuity of $P_0$ in its domain and \eqref{robust} that
    \begin{equation}\label{eq:cc_0}
      \bm \eta^*\in \partial^\infty P_0(\bm x^*).
    \end{equation}
    Next, we prove that $\bm\chi_i^*\in\A_i^*\partial^{\infty}P_i(\A_i\bm x^*)$ for $i=1,...,m$. To proceed, we define for each $i=1,\ldots,m$,
    \[
    w_i^t:= \left\|\frac{\A_i\bm x^t - \bm \zeta_i^t}{\lambda_{i,t-1}r_t}\right\|
    \]
    and claim that $\{w_i^t\}_{t\in\cal I}$ is bounded for all $i=1,\ldots,m$. For an arbitrarily fixed $i\in\{1,...,m\}$, suppose to the contrary that $\{w_i^t\}_{t\in\cal I}$ is unbounded and we assume without loss of generality that $\lim\limits_{t\in \cal I}w_i^t = \infty$ and that
    \begin{equation}\label{w_bd}
    \lim\limits_{t\in \cal I}\frac1{w_i^t}\frac{\A_i\bm x^t - \bm \zeta_i^t}{\lambda_{i,t-1}r_t} = \bm \psi_i^*
    \end{equation}
    for some $\bm \psi_i^*$ with unit norm.
    Then from the second equation in \eqref{proof1rel3}, we have
    \begin{equation}\label{psi_0}
      \|\bm \psi_i^*\| = 1\ \ {\rm and}\ \ \A_i^*\bm \psi_i^* =\bm 0.
    \end{equation}
    In addition, we observe from \eqref{w_bd} that
    \[
    \bm \psi_i^* = \lim\limits_{t\in \cal I}\frac1{w_i^t}\frac{\A_i\bm x^t - \bm \zeta_i^t}{\lambda_{i,t-1}r_t}\in \left\{\lim_{t\in \cal I}\frac{1}{w_i^tr_t}\bm u^t:\; \bm u^t \in \partial P_i(\bm \zeta_i^t) \mbox{ for each $t$} \right\} \subseteq \partial^\infty P_i(\A_i\bm x^*).
    \]
     where the first inclusion follows from \eqref{key_inclusion} and the second inclusion follows from Lemma~\ref{prox_limit} (so that $\lim\limits_{t\in \cal I}\bm\zeta^t_i= \A_i\bm x^*$ and $\{\bm \zeta^t_i\}_{t\in \cal I}\subseteq {\rm dom}\,P_i$), the continuity of $P_i$ in its domain and \eqref{robust}. These together with the facts $\bm 0\in\partial^{\infty}P_0(\bm x^*)$ , $\bm 0 \in\partial^{\infty}P_i(\A_i\bm x^*)$ ($i=1,...,m$)\footnote{These follow from (i) and \cite[Corollary~8.10]{RockWets98}.} and \eqref{psi_0} contradict \eqref{CQ}. Consequently, $\{w_i^t\}_{t\in\cal I}$ is bounded for all $i=1,\ldots,m$. Then, without loss of generality, we assume that $\lim\limits_{t\in \cal I}\frac{\A_i\bm x^t - \bm \zeta_i^t}{\lambda_{i,t-1}r_t}$ exists for all $i=1,...,m$. Then, for each $i=1,\ldots,m$, we observe from \eqref{proof1rel3} that
    \[
    \bm \chi_i^* = \A_i^*\lim\limits_{t\in \cal I}\frac{\A_i\bm x^t - \bm \zeta_i^t}{\lambda_{i,t-1}r_t} \in \A_i^*\left\{\lim_{t\in \cal I}\frac{1}{r_t}\bm u^t:\; \bm u^t \in \partial P_i(\bm \zeta_i^t) \mbox{ for each $t$} \right\} \subseteq \A_i^*\partial^\infty P_i(\A_i\bm x^*),
    \]
    where the first inclusion follows from \eqref{key_inclusion} and the second inclusion follows from Lemma~\ref{prox_limit} (so that $\lim\limits_{t\in \cal I}\bm\zeta^t_i= \A_i\bm x^*$ and $\{\bm \zeta^t_i\}_{t\in \cal I}\subseteq {\rm dom}\,P_i$ for each $i=1,\ldots,m$), the continuity of $P_i$ in its domain and \eqref{robust}.
    These together with \eqref{proof1rel2}, \eqref{sum} and \eqref{eq:cc_0} contradict \eqref{CQ}. Consequently, $\{r_t\}_{t\in \cal I}$ is bounded.

    Since $\{r_t\}_{t\in \cal I}$ is bounded, we may assume without loss of generality that
    \begin{equation}\label{eta}
    \lim\limits_{t\in \cal I} \bm \eta^t  =\tilde{\bm \eta}^*\ \ {\rm and}\ \  \lim\limits_{t\in \cal I}\frac1{\lambda_{i,t-1}}\A_i^*(\A_i\bm x^t - \bm \zeta_i^t) = \tilde {\bm \chi}^*_i
    \end{equation}
    for some $\tilde{\bm \eta}^*$ and $\tilde{\bm \chi}^*_i$, $i=1,\ldots,m$.
    Then we have from \eqref{robust} and the continuity of $P_0$ in its domain that
    \begin{equation}\label{haha1}
      \tilde{\bm \eta}^*\in \partial P_0(\bm x^*).
    \end{equation}
    Next, we prove that $\tilde {\bm \chi}^*_i\in\A_i^*\partial P_i(\A_i\bm x^*)$ for $i=1,...,m$. To proceed, we define for each $i=1,\ldots,m$,
    \[
    \nu_i^t:= \left\|\frac{\A_i\bm x^t - \bm \zeta_i^t}{\lambda_{i,t-1}}\right\|
    \]
    and claim that $\{\nu_i^t\}_{t\in\cal I}$ is bounded for all $i=1,\ldots,m$. For an arbitrary fixed $i\in\{1,...,m\}$, suppose to the contrary that $\{\nu_i^t\}_{t\in\cal I}$ is unbounded and we assume without loss of generality that $\lim\limits_{t\in \cal I}\nu_i^t = \infty$ and that
    \begin{equation}\label{nu_bd}
    \lim\limits_{t\in \cal I}\frac1{\nu_i^t}\frac{\A_i\bm x^t - \bm \zeta_i^t}{\lambda_{i,t-1}} = \bm\phi_i^*
    \end{equation}
    for some $\bm \phi^*_i$ with unit norm.
    Notice from the second equation of \eqref{eta} that
    \begin{equation}\label{psi}
      \|\bm \phi_i^*\| = 1\ \ {\rm and}\ \ \A_i^*\bm \phi_i^* =\bm 0.
    \end{equation}
    In addition, we observe from \eqref{nu_bd} that
    \[
    \bm\phi_i^* = \lim\limits_{t\in \cal I}\frac1{\nu_i^t}\frac{\A_i\bm x^t - \bm \zeta_i^t}{\lambda_{i,t-1}}\in \left\{\lim_{t\in \cal I}\frac{1}{\nu_i^t}\bm u^t:\; \bm u^t \in \partial P_i(\bm \zeta_i^t) \mbox{ for each $t$} \right\} \subseteq \partial^\infty P_i(\A_i\bm x^*).
    \]
     where the first inclusion follows from \eqref{key_inclusion} and the second inclusion follows from Lemma~\ref{prox_limit} (so that $\lim\limits_{t\in \cal I}\bm\zeta^t_i=\A_ix^*$ and $\{\bm \zeta^t_i\}_{t\in \cal I}\subseteq {\rm dom}\,P_i$), the continuity of $P_i$ in its domain and \eqref{robust}. These together with the facts $\bm 0\in\partial^{\infty}P_0(\bm x^*)$, $\bm 0 \in\partial^{\infty}P_i(\A_i\bm x^*)$ ($i=1,...,m$)\footnote{These follow from (i) and \cite[Corollary~8.10]{RockWets98}.} and \eqref{psi} contradict \eqref{CQ}. Consequently, $\{\nu_i^t\}_{t\in\cal I}$ is bounded for all $i=1,\ldots,m$. Then, without loss of generality, we assume that $\lim\limits_{t\in \cal I}\frac{\A_i\bm x^t - \bm \zeta_i^t}{\lambda_{i,t-1}}$ exists for all $i=1,...,m$. Therefore, for each $i=1,\ldots,m$, we obtain from \eqref{eta} that
    \begin{equation}\label{haha2}
    \tilde {\bm \chi}^*_i = \A_i^*\lim\limits_{t\in \cal I}\frac{\A_i\bm x^t - \bm \zeta_i^t}{\lambda_{i,t-1}} \in \A_i^*\left\{\lim_{t\in \cal I}\bm u^t:\; \bm u^t \in \partial P_i(\bm \zeta_i^t) \mbox{ for each $t$} \right\} \subseteq \A_i^*\partial P_i(\A_i\bm x^*),
    \end{equation}
    where the first inclusion follows from \eqref{key_inclusion} and the second inclusion follows from Lemma~\ref{prox_limit} (so that $\lim\limits_{t\in \cal I}\bm\zeta^t_i= \A_i\bm x^*$ and $\{\bm \zeta^t_i\}_{t\in \cal I}\subseteq {\rm dom}\,P_i$ for each $i=1,\ldots,m$), the continuity of $P_i$ in its domain and \eqref{robust}. Passing to the limit in \eqref{condition2} along $t\in \cal I$ and invoking \eqref{eta}, \eqref{haha1} and \eqref{haha2}, we see that
    \[
    0 = \nabla f(\bm x^*) + \tilde {\bm \eta}^* + \sum_{i=1}^m\tilde {\bm \chi}_i^*\in \nabla f(\bm x^*)+ \partial P_0(\bm x^*) + \sum_{i=1}^m \A_i^*\partial P_i(\A_i\bm x^*).
    \]
    This completes the proof.\qedwhite
\end{proof}

\begin{REMA}
\normalfont  If, instead of \eqref{condition}, one can guarantee that
  \[
  F_{\bm \lambda_t}(\bm x^{t,l_t}) \le \inf F_{\bm \lambda_t} + \epsilon_t,
  \]
  then one can show that any accumulation point of the sequence $\{\bm x^t\}$ generated by SDCAM is a global minimizer of \eqref{sparseprob}. To see this, recall from \cite[Theorem~1.25]{RockWets98} that $e_{\lambda_{i,t}}P_i(\A_i\bm x)\to P_i(\A_i\bm x)$ for each $i$ and all $\bm x$, and from the discussion on \cite[Page~244]{RockWets98} that $\{(e_{\lambda_{i,t}}P_i)\circ \A_i\}$ epiconverges to $P_i\circ \A_i$ for each $i$. Using these together with \cite[Theorem~7.46]{RockWets98}, we further see that $\{F_{\bm \lambda_t}\}$ epiconverges to $F$. Now, in view of \cite[Theorem~7.31(b)]{RockWets98}, we conclude that any accumulation point of the sequence $\{\bm x^t\}$ generated by SDCAM is a global minimizer of $F$.
\end{REMA}

\section{Applications to structured optimization problems}\label{sec:struct}

\subsection{Problems involving sparsity}\label{sec:cardinality}

Consider the following $\ell_0$-constrained optimization problem discussed in \cite{tono2017}:
\begin{align}
\begin{array}{cl}
    \underset{\bm{x}}{\mbox{minimize}} & f(\bm{x})\\
    \mbox{subject to} & \|\bm{x}\|_0\leq k,~\bm{x}\in C,
\end{array}
  \label{eq:l0const-1}
\end{align}
where $f$ is as in \eqref{sparseprob} and $C$ is a nonempty closed set.
This model includes many important application problems such as sparse principal component analysis,
sparse portfolio selection and sparse nonnegative linear regression as special cases.
These applications typically involve a closed set $C$ whose projection is easy to compute. For instance, we have $f(\bvec x)=-\bvec x^\top \bm V\bvec x,$ defined with a covariance matrix
$\bvec V \in {\cal S}^n$ and
$C = \{\bm x:\; \|\bm x\|=1\}$ for sparse principal component analysis ~\cite{thiao2010dc}.
As another example, for sparse nonnegative linear regression \cite{slawski2013non}, $f(\bvec x)=\frac{1}{2}\|\bvec A\bvec x-\bvec b\|^2$
defined with $\bvec A\in \mathbb R^{m\times n}$ and $\bvec b\in\mathbb R^m$,
and $C = \{\bm x:\; \bm x\ge 0\}$ are used. For these two examples, the direct projection onto $C\cap \{\bm x:\; \|\bm x\|_0\le k\}$ is easy to compute, and the proximal gradient algorithm can then be applied to solving \eqref{eq:l0const-1}.

We next discuss a specific example where the direct projection onto $C\cap \{\bm x:\; \|\bm x\|_0\le k\}$ might not be easy to compute, and describe how our SDCAM can be applied.

\begin{EXAMP}
  [Sparse portfolio problem]\label{examp:portfolio}

\normalfont Given a basket of investable assets,
the Markowitz model \cite{Markowitz52} seeks to find the optimal asset allocation of the portfolio by minimizing the estimated variance with an expected return above a specified level. More recently, \cite{Brodie2009} has added the $\ell_1$-norm to the classical Markowitz model to obtain sparse portfolios, and after that, various types of sparse regularizers such as $\ell_p$-norm $(0 < p < 1)$ are incorporated into the Markowitz model
(e.g., \cite{Chen13}).

The sparse portfolio selection problem we consider here takes the following form:
\begin{align}
\begin{array}{cl}
    \underset{\bm{x}}{\mbox{minimize}} & f(\bm{x}) :=\frac{1}{2} \bm{x}^\top \bm Q \bm{x} \\
    \mbox{subject to} & \|\bm{x}\|_0\leq k, ~\bm x \ge 0,\;  \bm \e^\top \bm x = 1, \;\bm r^\top \bm x = r_0,
\end{array}
  \label{eq:l0const}
\end{align}
where ${\bm Q} \in \SC^n$ is the estimated covariance matrix of the portfolio, $\bm r \in \R^n$ is the estimated mean return vector of investable assets, $r_0 \in \R$ is a specific return level, and $\bm e$ is the vector of all ones. % with a matching dimension.
The constraint $\bm x \ge 0$ is known as
the non-shortsale constraint, and model \eqref{eq:l0const} is the formulation of the shorting-prohibited
sparse Markowitz model. We assume here that the feasible set of \eqref{eq:l0const} is nonempty.

Notice that the feasible set of \eqref{eq:l0const} is compact and hence \eqref{eq:l0const} has a solution. Let $\bm x^*$ be a solution of \eqref{eq:l0const} and $\tau \gg\max\limits_i|x^*_i|$. Define $\Omega :=\{\bm x: \|\bm x\|_0\le k, 0\leq \bm x \leq \tau\}$ and $S:= \{\bm x : \bm \e^\top \bm x = 1, \bm r^\top \bm x = r_0 \}$. Then \eqref{eq:l0const} can be rewritten in the form of \eqref{sparseprob} (with the same optimal value) as follows
\begin{equation}
  \underset{\bm{x} }{\mbox{minimize}}  ~ \displaystyle f(\bm x) + \underbrace{\delta_{\Omega}(\bm x)}_{P_0(\bm x)} + \underbrace{\delta_{S}(\bm x)}_{P_1(\bm x)}, \label{spp}
\end{equation}
in which $f+P_0$ is level-bounded. Therefore, we can apply SDCAM in Section \ref{sec3} to \eqref{spp}, and in each subproblem of SDCAM we can use NPG$_{\rm major}$ to minimize $F_{\bm \lambda_t}$ as described in Theorem~\ref{thm0}. The method involves computing two projections $\bm\proj_{\Omega}$ and $\bm\proj_S$, which are easy to compute. Indeed, we have $\max\{\min\{\tilde  H_k(\bm y),\tau\},0\}\in \proj_{\Omega}(\bm y)$, where $\tilde H_k(\bm v)$ keeps any $k$ largest entries of $\bm v$ and sets the rest to zero.
\footnote{To see this, recall from \cite[Proposition~3.1]{LuZh13} that an element $\bm \zeta^*$ of $\proj_{\Omega}(\bm y)$ can be obtained as
\[
\zeta^*_i = \begin{cases}
  \tilde \zeta^*_i & {\rm if}\ i \in I^*,\\
  0 & {\rm otherwise},
\end{cases}
\]
where $\tilde \zeta^*_i = \argmin\{\frac12(\zeta_i - y_i)^2:\; 0\le \zeta_i\le \tau\} = \max\{\min\{y_i,\tau\},0\}$, and $I^*$ is an index set of size $k$ corresponding to the $k$ largest values of $\{\frac12 y_i^2 - \frac12(\tilde \zeta^*_i - y_i)^2\}_{i=1}^n = \{\frac12 y_i^2 - \frac12(\min\{\max\{y_i - \tau,0\},y_i\})^2\}_{i=1}^n$. Since the function $t\mapsto \frac12t^2 - \frac12(\min\{\max\{t - \tau,0\},t\})^2$ is nondecreasing, we can let $I^*$ correspond to any $k$ largest entries of $\bm y$.}
%Note that the computation of $\bm\proj_S$ is also easy.
\qedwhite
\end{EXAMP}

In statistics, $\ell_1$-norm regularizer has been used for inducing sparsity in variable selection problems; see Lasso
\cite{tibshirani1996}, which is an application of the $\ell_1$ penalty to linear regression.
A more general model of Lasso, the generalized Lasso \cite{TibshiraniTaylor2011}, has been proposed as
\begin{align*}
\begin{array}{cl}
    \underset{\bm{x} }{\mbox{minimize}} & \frac{1}{2}\|\bm A\bm x-\bm b\|^2 + c\|\bm D \bm x\|_1,
\end{array}
\end{align*}
where $\bm A\in \mathbb R^{m\times n}$ is a matrix of predictors, $\bm b\in\mathbb R^m$ is a response vector, $c \geq 0$ is a tuning parameter and
$\bm D\in\mathbb R^{d \times n}$ is a specified penalty matrix.
The term $\|\bm D \bm x\|_1$ can enforce certain structural sparsity on the coefficients in the solution.
For example, with an appropriate $\bm D$, $\|\bm D \bm x\|_1$ can express $\sum_{i=2}^n|x_i-x_{i-1}|$,
which penalizes the absolute differences in adjacent coordinates of $\bm x$. This specific $\bm D$ leads to the so-called fused Lasso. A variant of this type of regularizer (anisotropic total variation regularizer) is also used in image processing for minimizing the horizontal or/and vertical differences between pixels. Some other applications which require a non-identity matrix $\bm D$ in the generalized Lasso were discussed in \cite{TibshiraniTaylor2011}. In the next example, we discuss how our SDCAM can be applied to some nonconvex variants of the generalized Lasso problem.

\begin{EXAMP}
  [Nonconvex fused regularized problem]\label{examp:fusedreg}
  \normalfont Similarly as in \cite{ParakhSelesnick15}, we consider the following
  nonconvex fused regularized problem
    \begin{align}\label{nfrp}
\begin{array}{cl}
  \underset{\bm{x}}{\mbox{minimize}} & \frac{1}{2}\|\bm A\bm x-\bm b\|^2 + c_1\phi_1(\bm x)
+ c_2\phi_2(\bm D \bm x),
\end{array}
  \end{align}
  where $\bm A\in\R^{m\times n}$, $\bm b\in\R^m$, $\bm D \bm x = (x_2 - x_1,...,x_n - x_{n-1})^\top$, $c_1 > 0$ and $c_2 > 0$ are regularization parameters, $\phi_1(\bm x) = \sum_{i=1}^n \varphi_i(|x_i|)$ and $\phi_2$ are nonconvex sparsity-inducing regularizers with $\varphi_i: \R_+ \rightarrow \R_+$ being closed and nondecreasing, and $\phi_2: \R^{n-1}\rightarrow\R_+$ being closed and level-bounded.

  Note that \eqref{nfrp} can be rewritten in the form of
  \begin{align}\label{new_form}
\begin{array}{cl}
  \underset{\bm{x}}{\mbox{minimize}} & g(\bm {\tilde{A}}\bm x-\bm {\tilde{b}}) + \Psi(\bm x),
\end{array}
  \end{align}
  in which $\bm{\tilde{A}} = \begin{pmatrix}\bm A\\\bm D\end{pmatrix}$, $\bm{\tilde{b}} = \begin{pmatrix}\bm b\\\bm 0\end{pmatrix}$, $g(\bm y) = \frac12\|\bm y_1\|^2 + c_2\phi_2(\bm y_2)$ with $\bm y := (\bm y_1, \bm y_2) \in\R^m \times \R^{n-1}$, and $\Psi(\bm x) = c_1
\sum_{i=1}^n \varphi_i(|x_i|)$. It is routine to check that $g$ and $\Psi$ satisfy \cite[Assumption 2]{LuLi2015}. Hence, according to \cite[Theorem 2.1]{LuLi2015}, we know that \eqref{new_form}, and hence \eqref{nfrp}, has at least one solution.

Notice that we can directly apply the SDCAM in Section \ref{sec3} to \eqref{nfrp} when $\phi_1$ is level-bounded, e.g., $\phi_1(\bm x) = \|\bm x\|^p$: we set $f(\bm x) = \frac{1}{2}\|\bm A\bm x-\bm b\|^2$, $P_0 = c_1\phi_1$ and $P_1 = c_2\phi_2$ with $\A_1 = \bm D$ in this case. When the NPG$_{\rm major}$ is applied as described in Theorem~\ref{thm0} for solving the corresponding subproblems, it involves computing the proximal mappings $\prox_{\mu \phi_1}$ and $\prox_{\mu\phi_2}$ for $\mu > 0$. These are easy to compute for many well-known nonconvex sparse regularizers; see \cite{gong2013general}.

Finally, in the case when $\phi_1$ is not level-bounded, let $\bm x^*$ be a solution of \eqref{nfrp} and $\tau\gg\max\limits_i|x^*_i|$. We define $\Omega:=\{\bm x: \max\limits_i|x_i|\le\tau\}$ and rewrite \eqref{nfrp} in the form of \eqref{sparseprob} (with the same optimal value) as follows
    \begin{equation}
  \underset{\bm{x} }{\mbox{minimize}}  ~ \underbrace{\displaystyle \frac{1}{2}\|\bm A\bm x-\bm b\|^2}_{f(\bm x)} + \underbrace{c_1
\sum_{i=1}^n \varphi_i(|x_i|) + \delta_{\Omega}(\bm x)}_{P_0(\bm x)} + \underbrace{c_2\phi_2(\bm D \bm x)}_{P_1(\A_1\bm x)}. \label{FR}
\end{equation}
Then $f+P_0$ is level-bounded and hence the SDCAM in Section \ref{sec3} can be applied. When the NPG$_{\rm major}$ is applied in the subproblem of SDCAM as described in Theorem~\ref{thm0}, it involves computing the proximal mappings $\prox_{\mu P_0}$ and $\prox_{\mu\phi_2}$ for $\mu > 0$. Note that $\prox_{\mu P_0}$ can be obtained from $\prox_{\mu \psi_i}$ with $\psi_i(x_i) := c_1\varphi_i(|x_i|) + \delta_{|\cdot|\le\tau}(x_i)$, $i=1,...,n$, which can be efficiently computed for various nonconvex sparse regularizers such as SCAD, MCP, $\ell_p$ penalty and Capped-$\ell_1$ (see \cite{gong2013general}). Finally, the computation of $\prox_{\mu\phi_2}$ is also easy for many of these regularizers.
\qedwhite
\end{EXAMP}

\subsection{Problems with rank constraints}

Our algorithm can also be applied to rank-constrained nonconvex nonsmooth matrix optimization problems. We discuss some concrete examples below.

For notational simplicity, from now on, we let
\[
\Xi_k := \{\bm X:\; \rank(\bm X)\le k\}
\]
for a given integer $k$. Note that if $P_1 = \delta_{\Xi_k}$, then
\[
e_{\lambda_1} P_1(\bm X)=\frac{1}{2 \lambda_1} {\rm dist}^2(\bm X,\Xi_k)=\frac{1}{2 \lambda_1} (\|\bm X\|_F^2 - |\!|\!|\bm{X}|\!|\!|_{k,2}^2),
\]
where $|\!|\!|\bm{X}|\!|\!|_{k,2}^2$ denotes the sum of squares of the $k$ largest singular values of $\bm X$.
The function $\bm X \mapsto \|\bm X\|_F^2 - |\!|\!|\bm{X}|\!|\!|_{k,2}^2$ is a ``rank-related" variant of the so-called $k$-sparsity functions \cite{pang2017} because the relation $\mathrm{rank}(\bm{X}) \leq k$ can be equivalently expressed as
$\|\bm{X}\|_F^2 -|\!|\!|\bm{X}|\!|\!|_{k,2}^2=0$. A variant of this function was used in \cite{tono2017} as a penalty function for inducing sparsity. It is interesting to note that this function falls out naturally from the Moreau envelope of the indicator function of $\Xi_k$.

\begin{EXAMP}[Matrix completion]\label{examp:mc}
\normalfont The problem of recovering a low-rank data matrix $\bm{M} \in \R^{m \times n}$ from a sampling of its entries
is known as the matrix completion problem \cite{candes2009exact}. This problem can be formulated as
\begin{align*}
\begin{array}{cl}
\underset{\bm{X}}{\mbox{minimize}} & ~ \displaystyle \mathrm{rank}(\bm{X}) \\
\mbox{subject to}
& P_\Omega(\bm X) = P_\Omega(\bm M),
\end{array}
\end{align*}
where $\Omega$ is the index set of known entries of $\bm{M}$, and $P_\Omega$ is the sampling map defined as
\[
[P_\Omega(\bm Y)]_{ij} = \begin{cases}
  Y_{ij} & {\rm if} \ (i,j)\in \Omega,\\
  0 & {\rm otherwise}.
\end{cases}
\]
When the entries of the data matrix are noisy, one can consider the following variants of the above model:
\begin{align*}
\begin{array}{cl}
  \underset{\bm{X}}{\mbox{minimize}} &  \displaystyle \|P_\Omega(\bm X) - P_\Omega(\bm M)\|_F^2 \\
  \mbox{subject to} & \ \mathrm{rank}(\bm{X})\le k,
\end{array}\ {\rm or}\
\begin{array}{cl}
  \underset{\bm{X}}{\mbox{minimize}} &  \displaystyle  \|P_\Omega(\bm X) - P_\Omega(\bm M)\|_F^2 + \mu\, \mathrm{rank}(\bm{X}),
\end{array}
\end{align*}
where $\mu > 0$ is tuning parameter, and $k$ is a positive integer. Since these problems are nonconvex in general, some popular convex relaxation approaches have been proposed, where
the rank function is replaced by the nuclear norm function \cite{ReFaPa10}. The convex relaxations can be shown to be equivalent to the original nonconvex problems under suitable conditions \cite{candes2009exact}.

Here we consider the following variation of the matrix completion problem:
\begin{align}\label{MC1}
\begin{array}{cl}
\underset{\bm{X}}{\mbox{minimize}} & ~\displaystyle \frac12\|P_{\Omega}(\bm X) - P_{\Omega}(\bm M)\|_F^2 \\
\mbox{subject to}
& P_{\Theta}(\bm X) = P_{\Theta}(\bm M), ~\displaystyle \mathrm{rank}(\bm{X}) \le k,
\end{array}
\end{align}
where $\Omega$ is an index set corresponding to possibly {\em noisy} known entries of $\bm{M}$, and $\Theta$ is another index set corresponding to {\em noiseless} known entries of $\bm M$. Suppose that \eqref{MC1} has a solution $\bm X^*$, and take $\tau\gg \max\{\max\limits_{i,j}|X^*_{ij}|,\sigma_{\max}(\bm X^*)\}$.

Let $S:=\{\bm{X}:  P_{\Theta}(\bm X) = P_{\Theta}(\bm M)\}$, $\tilde S:=\{\bm{X}\in S: \max\limits_{i,j}|X_{ij}|\le \tau\}$ and $\tilde \Xi_k :=\{\bm X\in \Xi_k: \sigma_{\max}(\bm X)\le \tau\}$. Then \eqref{MC1} can be rewritten in the form of \eqref{sparseprob} (with the same optimal value) in the following two ways:
\begin{align}
\begin{array}{cl}
  \underset{\bm{X} }{\mbox{minimize}} & ~ \displaystyle \underbrace{\frac12\|P_{\Omega}(\bm X - \bm M)\|_F^2}_{f(\bm X)} + \underbrace{\delta_{S}(\bm X)}_{P_1(\bm X)} + \underbrace{\delta_{\tilde \Xi_k}(\bm X)}_{P_0(\bm X)}, \label{MC}
\end{array}\\
\begin{array}{cl}
  \underset{\bm{X} }{\mbox{minimize}} & ~ \displaystyle \underbrace{\frac12\|P_{\Omega}(\bm X - \bm M)\|_F^2}_{f(\bm X)} + \underbrace{\delta_{\tilde S}(\bm X)}_{P_0(\bm X)} + \underbrace{\delta_{\Xi_k}(\bm X)}_{P_1(\bm X)}. \label{MCMC}
\end{array}
\end{align}
Note that in both cases, $f+P_0$ is level-bounded and hence the SDCAM in Section~\ref{sec3} can be applied.

Suppose that SDCAM is applied to \eqref{MC}. Then when the NPG$_{\rm major}$ is applied as described in Theorem~\ref{thm0} for solving the subproblems, it requires computing $\proj_{S}$ and $\proj_{\tilde \Xi_k}$. Both of these are easy to compute. In particular, let $\bm U \Diag(\bm \sigma)\bm V^\top$ be a singular value decomposition of $\bm W$. Then an element $\bm Y \in \proj_{\tilde \Xi_k}(\bm W)$ can be computed as $\bm Y = \bm U \Diag(\bm \zeta^*)\bm V^\top$ with $\bm \zeta^* =\min\{H_k(\bm \sigma),\tau\bm e\}$, where $\bm e$ is the vector of all ones, the minimum is taken componentwise, and $H_k(\bm v)$ is the hard thresholding operator that keeps any $k$ largest entries of $\bm v$ in magnitude and sets the rest to zero.
\footnote{To see this, recall from \cite[Corollary~2.3]{LuZhLi15} and \cite[Proposition~3.1]{LuZh13} that an element $\bm Y \in \proj_{\tilde \Xi_k}(\bm W)$ can be computed as $\bm Y = \bm U \Diag(\bm \zeta^*)\bm V^\top$, where
\[
\zeta^*_i = \begin{cases}
  \tilde \zeta^*_i & {\rm if}\ i \in I^*,\\
  0 & {\rm otherwise},
\end{cases}
\]
where $\tilde \zeta^*_i = \argmin\{\frac12(\zeta_i - \sigma_i)^2:\; |\zeta_i|\le \tau\} = \min\{\sigma_i,\tau\}$, and $I^*$ is an index set of size $k$ corresponding to the $k$ largest values of $\{\frac12\sigma_i^2 - \frac12(\tilde \zeta^*_i - \sigma_i)^2\}_{i=1}^n = \{\frac12\sigma_i^2 - \frac12(\max\{0,\sigma_i - \tau\})^2\}_{i=1}^n$. Since $t\mapsto \frac12t^2 - \frac12(\max\{0,t - \tau\})^2$ is nondecreasing for nonnegative $t$, we can take $I^*$ to correspond to any $k$ largest singular values.}

On the other hand, when applying SDCAM to \eqref{MCMC} with the NPG$_{\rm major}$ as described in Theorem~\ref{thm0} applied to the subproblems, one needs to compute $\proj_{\tilde S}$ and $\proj_{\Xi_k}$. Again, both of these are easy to compute. In particular, let $\bm U \Diag(\bm \sigma) \bm V^\top$ be a singular value decomposition of $\bm W$. Then an element $\bm Y \in \proj_{\Xi_k}(\bm W)$ can be computed as $\bm Y = \bm U \Diag(H_k(\bm\sigma)) \bm V^\top$.
%Finally, to apply SDCAM, we also need to find a feasible point $\bm X^{\rm feas}$ for \eqref{MC1}. The difficulty of finding such a point depends on the structure of $\Theta$.
\qedwhite
\end{EXAMP}

\begin{EXAMP}
[Nearest low-rank correlation matrix]\label{examp:correlation matrix}
\normalfont Finding the nearest low-rank correlation matrix has important applications in finance; see \cite{BHR10,GaoSun2010}. The problem is often formulated as
\begin{align}\label{ncm}
\begin{array}{cl}
\underset{\bm{X}\in\SC^n}{\mbox{minimize}} & ~ \frac12\|\bm H\circ(\bm X - \bm M)\|_F^2 \\
\mbox{subject to}
& {\rm diag}(\bm X) = \bm e,\\
& \bm X \succeq 0,~\rank(\bm X) \le k,
\end{array}
\end{align}
where $\SC^n$ is the space of $n\times n$ symmetric matrices, $\bm H$ is a given nonnegative weight matrix, $\bm M$ is a given symmetric matrix and $\bm e$ is the vector of all ones, $k\ge 1$. In \cite{GaoSun2010}, the constraint $\rank(\bm X)\le k$ was rewritten equivalently as requiring the sum of the $n-k$ smallest eigenvalues equal zero. A penalty approach was then adopted to handle this latter equality constraint.

In the following, we describe how to solve \eqref{ncm} by the SDCAM in Section 3. Notice that for any $\bm{X}\in\SC^n$ satisfying ${\rm diag}(\bm X) = \bm e$ and $\bm X \succeq 0$, we have $\bm X \preceq n\,\bm I$. Thus, the feasible set of \eqref{ncm} is compact and hence \eqref{ncm} has a solution. Let $\bm X^*$ be a solution of \eqref{ncm} and $\tau \gg \max\{\max\limits_{i,j}|X^*_{ij}|,\lambda_{\max}(\bm X^*)\}$. Define
\begin{gather*}
  S:= \{\bm X\in \SC^n: {\rm diag}(\bm X) = \bm e\},~~ \tilde S:=\{\bm{X}\in S: \max\limits_{i,j}|X_{ij}|\le \tau\}, \\
  \Pi_k:=\{\bm X\succeq 0:\rank(\bm X)\le k\},~~ \tilde \Pi_k :=\{\bm X\in \Pi_k: \lambda_{\max}(\bm X)\le \tau\}.
\end{gather*}
Then \eqref{ncm} can be rewritten in the form of \eqref{sparseprob} (with the same optimal value) in the following two ways:
\begin{align}
\begin{array}{cl}
  \underset{\bm{X}\in\SC^n }{\mbox{minimize}} & ~ \displaystyle \underbrace{\frac12\|\bm H\circ(\bm X - \bm M)\|_F^2}_{f(\bm X)} + \underbrace{\delta_{S}(\bm X)}_{P_1(\bm X)} + \underbrace{\delta_{\tilde \Pi_k}(\bm X)}_{P_0(\bm X)}, \label{NCM_1}
\end{array}\\
\begin{array}{cl}
  \underset{\bm{X}\in\SC^n }{\mbox{minimize}} & ~ \displaystyle \underbrace{\frac12\|\bm H\circ(\bm X - \bm M)\|_F^2}_{f(\bm X)} + \underbrace{\delta_{\tilde S}(\bm X)}_{P_0(\bm X)} + \underbrace{\delta_{\Pi_k}(\bm X)}_{P_1(\bm X)}. \label{NCM_2}
\end{array}
\end{align}
Notice that in both cases, $f+P_0$ is level-bounded and hence we can apply the SDCAM in Section 3.

We first look at \eqref{NCM_1}. When the NPG$_{\rm major}$ as described in Theorem~\ref{thm0} is applied to the subproblems, one has to compute $\proj_S$ and $\proj_{\tilde \Pi_k}$. Both projections can be easily computed. In particular, let $\bm U \Diag(\bm \lambda) \bm U^\top$ be an eigenvalue decomposition of $\bm W\in\SC^n$. Then an element $\bm Y \in \proj_{\tilde \Pi_k}(\bm W)$ can be computed as $\bm Y = \bm U \Diag(\bm \zeta^*)\bm V^\top$ with $\bm \zeta^* = \max\{\min\{\tilde  H_k(\bm \lambda),\tau\},0\}$, where $\tilde H_k(\bm v)$ keeps any $k$ largest entries of $\bm v$ and sets the rest to zero.
\footnote{To see this, recall from \cite[Proposition~2.8]{LuZhLi15} and \cite[Proposition~3.1]{LuZh13} that an element $\bm Y \in \proj_{\tilde \Pi_k}(\bm W)$ can be computed as $\bm Y = \bm U \Diag(\bm \zeta^*)\bm V^\top$, where
\[
\zeta^*_i = \begin{cases}
  \tilde \zeta^*_i & {\rm if}\ i \in I^*,\\
  0 & {\rm otherwise},
\end{cases}
\]
where $\tilde \zeta^*_i = \argmin\{\frac12(\zeta_i - \lambda_i)^2:\; 0\le \zeta_i\le \tau\} = \max\{\min\{\lambda_i,\tau\},0\}$, and $I^*$ is an index set of size $k$ corresponding to the $k$ largest values of $\{\frac12\lambda_i^2 - \frac12(\tilde \zeta^*_i - \lambda_i)^2\}_{i=1}^n = \{\frac12\lambda_i^2 - \frac12(\min\{\max\{\lambda_i - \tau,0\},\lambda_i\})^2\}_{i=1}^n$. Since the function $t\mapsto \frac12t^2 - \frac12(\min\{\max\{t - \tau,0\},t\})^2$ is nondecreasing, we can let $I^*$ correspond to any $k$ largest entries of $\bm \lambda$.
}

We next turn to \eqref{NCM_2}. In this case, in each NPG$_{\rm major}$ iteration, one has to compute $\proj_{\tilde S}$ and $\proj_{\Pi_k}$. Again, both projections can be easily computed. In particular, let $\bm U \Diag(\bm \lambda) \bm U^\top$ be an eigenvalue decomposition of $\bm W\in\SC^n$. Then an element $\bm Y \in \proj_{\Pi_k}(\bm W)$ can be computed as $\bm Y = \bm U \Diag(\max\{\tilde H_k(\bm\lambda),\bm 0\})\bm U^\top$.
%Finally, to apply SDCAM, we also need to specify a feasible point $\bm X^{\rm feas}$ for \eqref{ncm}. It is easy to see that the matrix of all ones is feasible.
\qedwhite
\end{EXAMP}

\begin{EXAMP}[Simultaneously sparse and low rank matrix optimization problem]\label{examp:lr_sparse}
\normalfont The following problem was considered in \cite{richard2012}:
\[
\underset{\bm{X}}{\mbox{minimize}}  ~~ f(\bm{X}) + \gamma \|\vect(\bm{X})\|_1 + \tau \|\bm{X}\|_*,
\]
where $f$ is as in \eqref{sparseprob}, $\gamma$ and $\tau$ are positive numbers.
This problem aims at finding solutions which are both sparse and low-rank, and can be applied to identifying clusters in social networks; see \cite[Section~6.2]{richard2012}. This model relaxes and penalizes the sparsity index $\|\vect(\bm{X})\|_0$ and the low-rank index $\rank(\bm X)$ by two convex functions $\|\vect(\bm{X})\|_1$ and $\|\bm{X}\|_*$, respectively.

Here, we consider the following variant that explicitly incorporates the sparsity and rank constraints:
\begin{align}\label{simul_SL}
\begin{array}{cl}
\underset{\bm{X}}{\mbox{minimize}} & ~\displaystyle f(\bm{X}) \\
\mbox{subject to}
& \|\vect(\bm{X})\|_0 \leq s, ~\displaystyle \rank(\bm{X}) \leq k.
\end{array}
\end{align}
Suppose that \eqref{simul_SL} has a solution $\bm X^*$, and let $\tau \gg \max\{\max\limits_{i,j}|X^*_{ij}|,\sigma_{\max}(\bm X^*)\}$. Define $S:=\{\bm{X}:  \|\vect(\bm{X})\|_0 \leq s\}$, $\tilde S:=\{\bm{X}\in S: \max\limits_{i,j}|X_{ij}|\le \tau\}$ and $\tilde \Xi_k :=\{\bm X\in \Xi_k: \sigma_{\max}(\bm X)\le \tau\}$. Then \eqref{simul_SL} can be rewritten in the form of \eqref{sparseprob} (with the same optimal value) in the following two ways:
\begin{align}
\begin{array}{cl}
  \underset{\bm{X} }{\mbox{minimize}} & ~ \displaystyle f(\bm X) + \underbrace{\delta_{S}(\bm X)}_{P_1(\bm X)} + \underbrace{\delta_{\tilde \Xi_k}(\bm X)}_{P_0(\bm X)}, \label{SL_1}
\end{array}\\
\begin{array}{cl}
  \underset{\bm{X} }{\mbox{minimize}} & ~ \displaystyle f(\bm X) + \underbrace{\delta_{\tilde S}(\bm X)}_{P_0(\bm X)} + \underbrace{\delta_{\Xi_k}(\bm X)}_{P_1(\bm X)}. \label{SL_2}
\end{array}
\end{align}
Note that in both cases, $f+P_0$ is level-bounded and hence the SDCAM in Section~\ref{sec3} can be applied. When the NPG$_{\rm major}$ as described in Theorem~\ref{thm0} is applied to the corresponding subproblems, one has to compute $\proj_S$ and $\proj_{\tilde \Xi_k}$ for \eqref{SL_1}, and $\proj_{\tilde S}$ and $\proj_{\Xi_k}$ for \eqref{SL_2}. All these projections can be computed efficiently; see Examples~\ref{examp:portfolio} and~\ref{examp:mc}.
%Finally, to apply SDCAM, we also need to find a feasible point $\bm X^{\rm feas}$ for \eqref{simul_SL}. We can set $\bm X^{\rm feas} = \bm 0$.
\qedwhite
\end{EXAMP}

\section{Numerical experiments}\label{sec5}

In this section, we apply our SDCAM in Section \ref{sec3} with subproblems solved by NPG$_{\rm major}$ as described in Theorem~\ref{thm0} to an instance of Example~\ref{examp:fusedreg} and Example~\ref{examp:lr_sparse}: the nonconvex fused regularized problem and the simultaneously sparse and low rank matrix optimization problem. All numerical experiments are performed in Matlab R2016a on a 64-bit PC with an Intel(R) Core(TM) i7-6700 CPU (3.41GHz)
and 32GB of RAM.

%\subsection{Nonconvex fused regularized problem}\label{sec51}
\subsection{Nonconvex fused regularized problem: comparison against a solution method based on smoothing}\label{sec51}
We consider the following special instance of nonconvex fused regularized problem:
\begin{align}\label{nfrp_numer}
  \begin{array}{cl}
  \underset{\bm{x}}{\mbox{minimize}} & \frac{1}{2}\|\bm x-\bm b\|^2 + c_1\|\bm x\|_1
+ c_2\|\bm D \bm x\|_p^p,
  \end{array}
\end{align}
where $c_1 > 0$, $c_2 > 0$, $p = 0.5$, $\bm D \bm x = (x_2 - x_1,...,x_n - x_{n-1})^\top$, and $\bm b\in \mathbb R^n$ is the noisy measurement of a piecewise constant sparse signal.
Notice that the function $\|\cdot\|_1$ is level-bounded. We can directly apply SDCAM as described in Example \ref{examp:fusedreg} and solve the subproblems by NPG$_{\rm major}$. On the other hand, a commonly used technique for handling optimization problems involving $\ell_p$ penalty functions ($0<p<1$) is smoothing. Thus, in our experiments below, we compare SDCAM with a method based on smoothing, the smoothing nonmonotone proximal gradient method (sNPG), for solving \eqref{nfrp_numer}. In sNPG, we solve the following sequence of subproblems approximately by NPG (this is NPG$_{\rm major}$ applied to \eqref{P1} when $g = 0$):
\begin{equation*}
  \underset{\bm{x}}{\mbox{minimize}}  ~ \underbrace{\displaystyle \frac{1}{2}\|\bm x-\bm b\|^2 + c_2\sum_{i=1}^{n-1}\left((\bm D\bm x)_i^2 + \lambda_t^2\right)^{\frac{p}{2}}}_{f_t(\bm x)} + \underbrace{c_1\|\bm x\|_1}_{Q(\bm x)},
\end{equation*}
where $\lambda_t\downarrow0$ is the smoothing parameter. The approximate stationary point of $f_t+Q$ obtained is then used as initialization for minimizing $f_{t+1} + Q$.

\paragraph{\bf Data generation:} We first randomly generate a piecewise constant signal $\bm x\in\R^n$ using the following Matlab code:
\begin{verbatim}
J = randperm(10);I = sort(J(1:6),'ascend');x = zeros(n,1);
for i = 1:r
    if randn > 0
        x(n*I(i)/10 - 3*n/50 - randi(3) : n*I(i)/10) = randi(3);
    else
        x(n*I(i)/10 - 3*n/50 - randi(3) : n*I(i)/10) = -randi(3);
    end
end
\end{verbatim}
Then we let $\bm b = \bm{x} + \sigma\bm \xi$, where $\sigma > 0$ is a noise factor  and $\bm \xi$ has i.i.d. standard Gaussian entries. In our experiments, motivated by \cite{ParakhSelesnick15}, we choose $c_1 = c_2 = \sigma\sqrt{n}/40$. We shall see that this choice leads to reasonable recovery results in Figure~\ref{Recover}. We also set $\sigma = 0.1$, $n = 2000$, $4000$, $6000$, $8000$, $10000$.

\paragraph{\bf Parameter setting:} In SDCAM, we set $\lambda_t = 1/10^{t+1}$ and $\bm x^{\rm feas}$ to be the vector of all ones. In the NPG$_{\rm major}$ for solving the subproblems, we set $M = 4$, $L_{\max} = 10^8$, $L_{\min} = 10^{-8}$, $\tau = 2$, $c = 10^{-4}$, $L^0_{t,0}=1$ and for $l\ge 1$,
\[
L_{t,l}^0 = \max\left\{\min\left\{\frac{{\bm s^l}^\top \bm y^l}{\|\bm s^l\|^2},L_{\max}\right\},L_{\min}\right\},
\]
(which is the inverse of the so-called Barzilai-Borwein stepsize) where $\bm s^l = \bm x^{t,l}- \bm x^{t,l-1}$ and $\bm y^l = \nabla h(\bm x^{t,l}) - \nabla h(\bm x^{t,l-1})$. We initialize NPG$_{\rm major}$ at $\bm x^{\rm feas}$ and terminate it when the maximum number of iterations exceeds $10000$ or
\begin{equation*}
  \frac{\|\bm x^{t,l} - \bm x^{t,l-1}\|}{\max(\|\bm x^{t,l}\|,1)} < \epsilon_t/\bar L_{t,l-1}\ {\rm or}\ \frac{|F_{\bm \lambda_t}(\bm x^{t,l}) - F_{\bm \lambda_t}(\bm x^{t,l-1})|}{\max\{1,|F_{\bm \lambda_t}(\bm x^{t,l})|\}} < 10^{-12},
\end{equation*}
where $\epsilon_0 = 10^{-5}$ and $\epsilon_t = \max\{\epsilon_{t-1}/1.5, 10^{-6}\}$. On the other hand, in sNPG, we also let $\lambda_t = 1/10^{t+1}$ and solve the subproblems using NPG (i.e., NPG$_{\rm major}$ applied to \eqref{P1} with $g = 0$) with the same setting as described above, except that the $F_{\bm \lambda_t}$ above is replaced by $f_t+Q$ and for $l\ge 1$,
\[
L_{t,l}^0 = \begin{cases}
  \max\left\{\min\left\{\frac{{\bm s^l}^\top \bm y^l}{\|\bm s^l\|^2},L_{\max}\right\},L_{\min}\right\} & {\rm if}\ {\bm s^l}^\top \bm y^l > 10^{-12},\\
  \max\left\{\min\left\{\bar L_{t,l-1}/2,L_{\max}\right\},L_{\min}\right\} & {\rm otherwise}.
\end{cases}
\]
Finally, we terminate SDCAM when $\lambda_t < 10^{-9}$. And for a fair comparison, we consider two different termination criteria for sNPG: $\lambda_t < 10^{-7}$ (sNPG$_{-7}$) and $\lambda_t < 10^{-8}$ (sNPG$_{-8}$).

\paragraph{\bf Numerical results:} In Table~\ref{fused_tab}, we compare SDCAM, sNPG$_{-7}$ and sNPG$_{-8}$ in terms of the number of iterations (iter), CPU time (CPU) and the terminating function values (fval), averaged over 10 randomly generated instances. One can see that the terminating function values are comparable, and SDCAM is in general faster than sNPG$_{-8}$ and slower than sNPG$_{-7}$. Moreover, SDCAM outperforms the sNPG's slightly in terms of function values when the dimension is relatively low ($\le 4000$). To illustrate the ability to recover the original signal, we also plot the original signal, the noisy measurement $\bm b$ and the signals recovered by SDCAM and sNPG$_{-8}$ for a random instance with $n = 2000$ in Figure~\ref{Recover}.

\begin{table}[h]
\caption{Results for SDCAM, sNPG$_{-7}$ and sNPG$_{-8}$ for solving \eqref{nfrp_numer}.}
\label{fused_tab}
\centering
%\hspace{-0.8 cm}
\begin{tabular}{|c|c|c|c|c|c|c|c|c|c|}
\hline
\multirow{2}{*}{$n$} & \multicolumn{3}{|c|}{iter} & \multicolumn{3}{|c|}{CPU}  & \multicolumn{3}{|c|}{fval} \\
\cline{2-10}
& SDCAM &  sNPG$_{-7}$&  sNPG$_{-8}$ & SDCAM &  sNPG$_{-7}$&  sNPG$_{-8}$ & SDCAM &  sNPG$_{-7}$&  sNPG$_{-8}$ \\\hline
    2000 &   27796 &   18498 &   23700 &   5.8 &   5.2 &   8.5 & 1.77278e+02 & 1.77294e+02 & 1.77290e+02\\
    4000 &   41686 &   33465 &   43465 &  17.0 &  16.5 &  29.1 & 4.95918e+02 & 4.95929e+02 & 4.95923e+02\\
    6000 &   45573 &   34113 &   44113 &  25.6 &  22.2 &  38.2 & 8.49430e+02 & 8.49420e+02 & 8.49398e+02\\
    8000 &   49089 &   28984 &   38984 &  34.7 &  23.7 &  43.5 & 1.32155e+03 & 1.32160e+03 & 1.32153e+03\\
   10000 &   45320 &   37379 &   47379 &  45.5 &  38.1 &  63.1 & 1.65874e+03 & 1.65870e+03 & 1.65864e+03\\ \hline
\end{tabular}
\end{table}

%\vspace{-1.3 cm}
%\begin{figure}[h]
%\begin{center}
%  \caption{Recovery comparison for noisy signal.}\label{Recover}
%  \includegraphics[scale = 0.4]{Recover.jpg}
%\end{center}
%\end{figure}

\begin{figure}[!htb]
    \centering
    \begin{minipage}{.49\textwidth}
        \centering
        \includegraphics[width=1\linewidth, height=0.3\textheight]{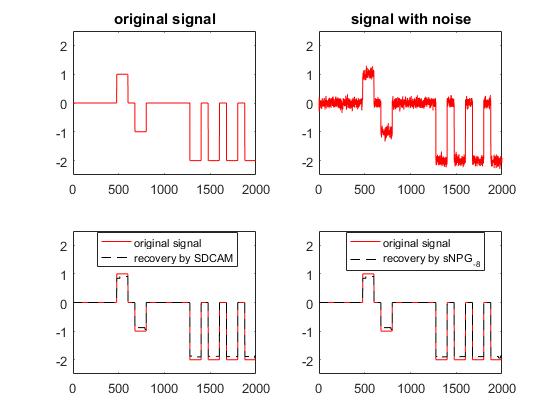}
        \caption{Recovery comparison for \\noisy signal.}
        \label{Recover}
    \end{minipage}%
    \begin{minipage}{0.49\textwidth}
        \centering
        \includegraphics[width=1\linewidth, height=0.3\textheight]{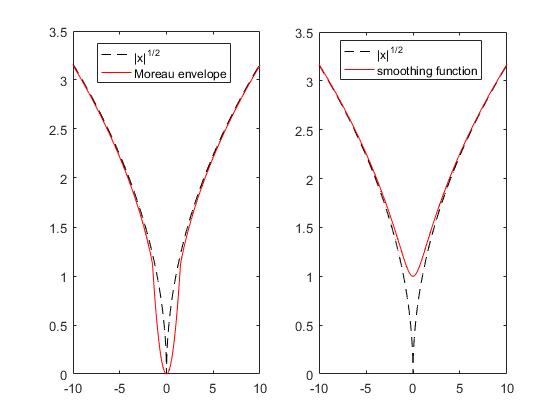}
        \caption{$|x|^{1/2}$ with its Moreau envelope and smoothing function.}
        \label{Envelope}
    \end{minipage}
\end{figure}

%\vspace{-0.5 cm}

To illustrate intuitively the approximation used in our SDCAM and sNPG, we plot the function $f(x) = |x|^{1/2}$ (in dashed lines), its Moreau envelope and its smoothing function in Figure~\ref{Envelope}. One can see that the envelope smooths the original nonsmooth point by a quadratic function. It is a lower approximation of $f$, while the smoothing function is an upper approximation of $f$.
%It can be seen that the lower approximation of $f$ is easier to reach the minimum solution than the approximation from the above for the minimization problem.
%\begin{figure}[h]
%\begin{center}
%  \caption{$|x|^{1/2}$ with its Moreau envelope and smoothing function.}\label{Envelope}
%  \includegraphics[scale = 0.20]{sm.jpg}
%\end{center}
%\end{figure}

%\subsection{Simultaneously sparse and low rank matrix optimization problem}
\subsection{Simultaneously sparse and low rank matrix optimization problem: which constraint should be modeled by $P_1$?}
We consider the following special instance of simultaneously sparse and low rank matrix optimization problem:
\begin{align}\label{rank_sparse}
\begin{array}{cl}
\underset{\bm{X}}{\mbox{minimize}} & ~\displaystyle \frac12 \|\bm{X} - \bm{M}\|_F^2 \\
\mbox{subject to}
& \|\vect(\bm{X})\|_0 \leq s, ~\displaystyle \rank(\bm{X}) \leq k,
\end{array}
\end{align}
where $\bm{M}\in\R^{m\times n}$ is a given noisy matrix, $s$ and $k$ are positive integers. Note that $f(\bm{X}): = \frac12 \|\bm{X} - \bm{M}\|_F^2$ is level-bounded. Therefore, \eqref{rank_sparse} has at least one solution. Then, as discussed in Example~\ref{examp:lr_sparse}, we can apply SDCAM to solving \eqref{rank_sparse} in two different ways by considering, respectively, the two formulations in \eqref{SL_1} and \eqref{SL_2}:\footnote{We would like to point out that we are indeed using $\Xi_k$ in place of $\tilde\Xi_k$ in \eqref{SL_1} and using $S$ in place of $\tilde S$ in \eqref{SL_2} in our experiments below. Notice that {\bf A3} is still satisfied because $f$ is level-bounded.}
the indicator function $\delta_{|\cdot|_0\le s}(\cdot)$ is approximated by the Moreau envelope
in \eqref{SL_1} and the function $\delta_{\rank(\cdot)\le k}(\cdot)$ is approximated by
its Moreau envelope in \eqref{SL_2}.
We call the method based on \eqref{SL_1} SDCAM$_r$ and the method based on \eqref{SL_2} SDCAM$_s$. In the following experiments, we compare these two methods.

\paragraph{\bf Data generation:} We first randomly generate $\bm{M}_1\in\R^{m\times k}$ and $\bm{M}_2\in\R^{k\times n}$ to have i.i.d. standard Gaussian entries. Then we set $m/10$ random rows of $\bm{M}_1$ to zero and let $\bm{M} = \bm{M}_1\bm{M}_2 + \sigma\bm\Delta$, where $\sigma > 0$ is a noise factor  and $\bm \Delta$ has i.i.d. standard Gaussian entries. We fix $n = 500$, $k = 10$ and $s = mn/10$, and we experiment with $\sigma = 0.005, 0.01,0.02$ and $m = 1000,2000,3000$ below.

\paragraph{\bf Parameter setting:} In both SDCAM$_r$ and SDCAM$_s$, we set $\lambda_t = 1/10^{t+1}$ and $\bm X^{\rm feas} = \bm 0$. In the NPG$_{\rm major}$ for solving the subproblems, we use the same parameter setting as in Section~\ref{sec51}.
%In the NPG$_{\rm major}$ for solving the subproblems, we set $M = 4$, $L_{\max} = 10^8$, $L_{\min} = 10^{-8}$, $\tau = 2$, $c = 10^{-4}$ and
%\[
%L_{t,l}^0 = \max\left\{\min\left\{\frac{{\rm tr}({\bm S^l}^\top \bm Y^l)}{\|\bm S^l\|_F^2},L_{\max}\right\},L_{\min}\right\},
%\]
%(which is the inverse of the so-called Barzalai-Borwein stepsize) where $\bm S^l = \bm X^{t,l}- \bm X^{t,l-1}$ and $\bm Y^l = \nabla h(\bm X^{t,l}) - \nabla h(\bm X^{t,l-1})$. We terminate NPG$_{\rm major}$ when the maximum number of iterations exceeds $10000$ or
%\begin{equation*}
%  \frac{\|\bm X^{t,l_t} - \bm X^{t,l_{t}-1}\|_F}{\max(\|\bm X^{t,l_t}\|_F,1)} < \epsilon_t/\bar L_{t,l_t-1} \ {\rm or}\ \frac{|F_{\bm \lambda_t}(\bm X^{t,l_t}) - F_{\bm \lambda_t}(\bm X^{t,l_t-1})|}{\max\{1,|F_{\bm \lambda_t}(\bm X^{t,l_t})|\}} < 10^{-12},
%\end{equation*}
%where $\epsilon_0 = 10^{-5}$ and $\epsilon_t = \max\{\epsilon_{t-1}/1.5, 10^{-6}\}$.
We initialize both algorithms at $\bm X^{\rm feas}$ and terminate them when
\begin{equation*}
  {\rm dist}(\bm X^t, S) \le 10^{-6}\cdot\|\bm X^t\|_F~~ {\rm and}~~\ {\rm dist}(\bm X^t, \Xi_k) \le 10^{-6}\cdot\|\bm X^t\|_F,
\end{equation*}
respectively.

\paragraph{\bf Numerical results:} In Table \ref{lowrank_sparse}, we compare SDCAM$_r$ and SDCAM$_s$ in terms of the number of iterations (iter), CPU time (CPU) and the feasibility violation (vio) (i.e., ${\rm dist}(\bm X^t, S)$ and ${\rm dist}(\bm X^t, \Xi_k)$, respectively) at termination, averaged over 10 randomly generated instances. One can see that SDCAM$_r$ takes fewer iterations and less time.
An intuitive explanation could be that the rank constraint is a more complicated constraint than the sparsity constraint to approximate via ``subgradients". Thus, the algorithm SDCAM$_r$ that maintains all its iterates in the rank constraint and then attempts to approximately satisfy the sparsity constraint as the algorithm progresses ends up converging more quickly.

%ively linearizes the concave part $-D_{\lambda_1,P_1}$ of
%the Moreau envelope \eqref{moreauEnv} of $P_1=\delta_{|\cdot|_0\le s}$ in SDCAM$_r$, or $P_1=\delta_{\rank(\cdot)\le k}$ in SDCAM$_s$. %\eqref{SL_2}.
%It is conceivable that the tightness of the linearization at NPG$_{\rm major}$ might lead to the difference in the number of iterations of SDCAM.

\begin{table}[h]
\caption{Comparison of SDCAM$_r$ and SDCAM$_s$ for solving \eqref{rank_sparse}.}
\label{lowrank_sparse}
\centering
%\hspace{-0.55 cm}
\begin{tabular}{|c|c|c|c|c|c|c|c|}
\hline
\multirow{2}{*}{$\sigma$} & \multirow{2}{*}{$m$} & \multicolumn{2}{|c|}{iter} & \multicolumn{2 }{|c|}{CPU}  & \multicolumn{2}{|c|}{vio}\\
\cline{3-8}
&  & SDCAM$_r$ & SDCAM$_s$  & SDCAM$_r$ &  SDCAM$_s$&  SDCAM$_r$ & SDCAM$_s$  \\\hline
       & 1000 &    41 &  5597 &   4.7 & 378.1 & 4.7569e-04 & 1.0515e-04\\
 0.005 & 2000 &    12 &  5298 &   4.0 & 647.0 & 6.7084e-04 & 1.5247e-04\\
       & 3000 &    12 &  4618 &   6.0 & 862.8 & 8.2038e-04 & 1.8857e-04\\ \hline
       & 1000 &  4508 &  7900 & 379.3 & 529.2 & 9.4347e-05 & 2.1032e-04\\
 0.010 & 2000 &  4453 &  7526 & 653.6 & 912.6 & 1.3412e-04 & 3.0580e-04\\
       & 3000 &  4428 &  5721 & 969.5 & 1080.6 & 1.6434e-04 & 3.7701e-04\\ \hline
       & 1000 &  4922 & 11631 & 413.7 & 769.2 & 1.8985e-04 & 4.2222e-04\\
 0.020 & 2000 &  4634 & 10267 & 675.5 & 1251.3 & 2.6849e-04 & 6.1136e-04\\
       & 3000 &  4580 & 10859 & 1003.5 & 2043.0 & 3.2804e-04 & 7.5510e-04\\ \hline
\end{tabular}
\end{table}

\section{Conclusions}\label{sec6}
In this paper, we propose a successive difference-of-convex approximation method for solving \eqref{sparseprob}. The key idea of this method is to approximate the nonsmooth functions in the objective of \eqref{sparseprob} by their Moreau envelopes. The approximation function can then be minimized by various proximal gradient methods with majorization techniques such as NPG$_{\rm major}$ in the appendix,
thanks to \eqref{moreauEnv}. We prove that the sequence generated by our method is bounded and any accumulation point is a stationary point of \eqref{sparseprob} under suitable conditions. We also discuss how to apply our method to concrete applications and conduct numerical experiments to illustrate its efficiency.

\appendix

\section{Convergence of an NPG method with majorization}
In this appendix, we consider the following optimization problem:
\begin{equation}\label{P1}
  \begin{array}{rl}
    \min\limits_{\bm x} & F(\bm x) = h(\bm x) + P(\bm x) - g(\bm x),
  \end{array}
\end{equation}
where $h$ is an $L_h$-smooth function, $P$ is a proper closed function with $\inf P > -\infty$ and $g$ is a continuous convex function. We assume in addition that there exists $\bm x^0\in {\rm dom}\,P$ so that $F$ is continuous in $\Omega(\bm x^0):= \{\bm x:\; F(\bm x)\le F(\bm x^0)\}$ and the set $\Omega(\bm x^0)$ is compact. As a consequence, it holds that $\inf F > -\infty$.

In Algorithm \ref{alg_npg} below, we describe an algorithm, the nonmonotone proximal gradient method with majorization (NPG$_{\rm major}$), for solving \eqref{P1}. We first show that the line-search criterion is well-defined.

\begin{algorithm}
\caption{The NPG$_{\rm major}$ for \eqref{P1}}\label{alg_npg}
\begin{description}
\item[\bf Step 0.] Choose $\bm x^0\in {\rm dom}\,P$ so that $\Omega(\bm x^0)$ is compact and $F$ is continuous in it. Pick $L_{\max}\ge L_{\min} >0$, $\tau>1$, $c>0$ and
an integer $M \ge 0$ arbitrarily. Set $t = 0$.

\item[\bf Step 1.]
Choose any $L^0_t \in [L_{\min},L_{\max}]$ and set $L_t = L^0_t$.
\begin{enumerate}
\item[\bf 1a)] Pick any $\bm \zeta^t\in \partial g(\bm x^t)$. Solve the subproblem
\begin{equation} \label{prox-subprob}
\bm u \in \Argmin_{\bm x} \left\{(\nabla h(\bm x^t) - \bm \zeta^t)^\top (\bm x-\bm x^t) +
\frac{L_t}{2}\|\bm x-\bm x^t\|^2 + P(\bm x)\right\}.
\end{equation}
\item[\bf 1b)] If
\begin{equation} \label{descent}
F(\bm u) \le \max_{[t-M]_+ \le i \le t} F(\bm x^i) - \frac{c}{2} \|\bm u-\bm x^t\|^2
\end{equation}
is satisfied, then go to {\bf step 2)}.
\item[\bf 1c)] Set $L_t \leftarrow \tau L_t$ and go to \bf step 1a).
\end{enumerate}

\item[\bf Step 2.] If a termination criterion is not met, set $\bar L_t = L_t$, $\bm x^{t+1}=\bm  u$, $t = t+1$. Go to \bf Step~1.
\end{description}
\end{algorithm}

\begin{PROP}\label{prop:bd}
  For each $t$, the condition \eqref{descent} is satisfied after at most
  \[
  \tilde n := \max\left\{\left\lceil \frac{\log(L_h + c) - \log(L_{\min})}{\log\tau}\right\rceil,1\right\}
  \]
  inner iterations, which is independent of $t$. Consequently, $\{\bar L_t\}$ is bounded.
\end{PROP}

\begin{proof}
  For each $t$ and $L > 0$, let $\bm u^t_L$ be an arbitrarily fixed element in
  \[
  \Argmin_{\bm x} \left\{(\nabla h(\bm x^t) - \bm \zeta^t)^\top (\bm x-\bm x^t) +
\frac{L}{2}\|\bm x-\bm x^t\|^2 + P(\bm x)\right\}.
  \]
  Then we have
  \[
  \begin{split}
    F(\bm u^t_L)& \le h(\bm x^t) + \nabla h(\bm x^t)^\top(\bm u^t_L - \bm x^t) + \frac{L_h}{2}\|\bm u^t_L - \bm x^t\|^2 + P(\bm u^t_L) - g(\bm x^t) - {\bm \zeta^t}^\top(\bm u^t_L - \bm x^t)\\
    & = h(\bm x^t) - g(\bm x^t) + (\nabla h(\bm x^t)- {\bm \zeta^t})^\top(\bm u^t_L - \bm x^t) + \frac{L_h}{2}\|\bm u^t_L - \bm x^t\|^2 + P(\bm u^t_L)\\
    & \le F(\bm x^t) + \frac{L_h-L}{2}\|\bm u^t_L - \bm x^t\|^2,
  \end{split}
  \]
  where the first inequality holds because of the $L_h$-smoothness of $h$, the convexity of $g$ and the fact that $\bm \zeta^t\in \partial g(\bm x^t)$, and the last inequality follows from the definition of $\bm u^t_L$ as a minimizer. Thus, at the $t$-th iteration, the criterion \eqref{descent} is satisfied by $\bm u = \bm u^t_L$ whenever $L \ge L_h + c$. Since we have
  \[
  \tau^{\tilde n} L^0_t \ge \tau^{\tilde n} L_{\min} \ge L_h+c,
  \]
  we conclude that \eqref{descent} must be satisfied at or before the $\tilde n$-th inner iteration. Consequently, we have $\bar L_t \le \tau^{\tilde n}L_{\max}$ for all $t$.\qedwhite
\end{proof}

The convergence of NPG$_{\rm major}$ can now be proved similarly as in \cite[Lemma~4]{WrightNF09}.

\begin{PROP}\label{prop:conv}
  Let $\{\bm x^t\}$ be the sequence generated by NPG$_{\rm major}$. Then $\|\bm x^{t+1} - \bm x^t\|\to 0$.
\end{PROP}

\end{document}